%% file: main.tex
\documentclass{article}
\usepackage{float}

\usepackage{enumerate}
\usepackage{natbib}
\usepackage{amsthm,amssymb, amsmath} 
\usepackage{dsfont}
\usepackage{color}
\usepackage{ifthen} 
\usepackage{xspace}
\usepackage{url}
\usepackage{braket}
\usepackage[title]{appendix}
\usepackage{graphicx}

\usepackage{geometry}
 \input{defsProb}

\def\Inte{{\theta(\widehat{F}_n)}}

\newcommand\Id{\operatorname{Id}}

\newcommand{\bigOp}{O_p}
\newcommand{\Appendice}{Appendix}

\usepackage[colorlinks=true]{hyperref}

\begin{document}

\title{Kaplan-Meier V- and U-statistics.}

\author{Tamara Fern\'andez\\
  University College London\\
  \texttt{t.a.fernandez@ucl.ac.uk} \\
   \and
   Nicol\'as Rivera \\
  University of Cambridge\\
   \texttt{nicolas.rivera@cl.cam.ac.uk} \\
}

\maketitle

\numberwithin{equation}{section}
\begin{abstract}
\sloppy{In this paper, we study Kaplan-Meier V- and  U-statistics respectively defined as $\theta(\widehat{F}_n)=\sum_{i,j}K(X_{[i:n]},X_{[j:n]})W_iW_j$ and $\theta_U(\widehat{F}_n)=\sum_{i\neq j}K(X_{[i:n]},X_{[j:n]})W_iW_j/\sum_{i\neq j}W_iW_j$, where $\widehat{F}_n$ is the Kaplan-Meier estimator, $\{W_1,\ldots,W_n\}$ are the Kaplan-Meier weights and $K:(0,\infty)^2\to\R$ is a symmetric kernel. As in the canonical setting of uncensored data, we differentiate between two asymptotic behaviours for $\theta(\widehat{F}_n)$ and $\theta_U(\widehat{F}_n)$. Additionally, we derive an asymptotic canonical V-statistic representation of the Kaplan-Meier V- and U-statistics. By using this representation we study properties of the asymptotic distribution. Applications to hypothesis testing are given.}
\end{abstract}

\section{Introduction}
Let $F$ be a distribution on $(0,\infty)$ of interest. In this paper we study the estimation of parameters of the form $\theta(F)=\int_0^{\infty}\int_0^{\infty} K(x,y)dF(x)dF(y)$, where $K:(0,\infty)^2\to\R$ is a measurable and symmetric  function, commonly known as kernel function. If we consider i.i.d. samples $X_1,\ldots, X_n$ from $F$, the standard estimators for $\theta(F)$ are the canonical V- and U-statistics, 
\begin{align}
\theta(\tilde{F}_n)=
\frac{1}{n^2}\sum_{i=1}^n\sum_{j=1}^nK(X_i,X_j), \quad
\text{and}\quad&
\theta_U(\tilde{F}_n)=\frac{1}{n(n-1)}\sum_{i\neq j}K(X_i,X_j),\nonumber
\end{align}
respectively, where  $\tilde{F}_n(t)=\frac{1}{n}\sum_{i=1}^n\ind_{\{X_i\leq t\}}$ denotes the empirical distribution.

Under some regularity conditions $\theta(\tilde{F}_n)\to \theta(F)$ and $\theta_U(\tilde{F}_n)\to \theta(F)$ almost surely as n approaches infinity. It is of interest to study the limit distribution of the differences $\theta(\tilde{F}_n)-\theta(F)$ and $\theta_U(\tilde{F}_n)-\theta(F)$. The standard theory of V- and U-statistics distinguishes  between two asymptotic behaviours for the distribution of the errors $\theta_U(\tilde{F}_n)-\theta(F)$ and $\theta(\tilde{F}_n)-\theta(F)$ when the number of samples tends to infinity. These two asymptotic regimes, widely known as \textit{degenerate} and \textit{non-degenerate}, are characterised by the behaviour of the variance of the projection $\phi:\R_+\to\R$, defined as
\begin{align}\label{eqn:random1294hsd9}
\phi(x) = \E(K(x, X_2)) = \int_0^{\infty} K(x,y)dF(y).
\end{align}
Assume that $\E(K(X_1,X_2)^2)<\infty$. In one hand, we say that we are in the 
non-degenerate regime if  $\var(\phi(X_1))>0$, where we have that
\begin{align*}
\sqrt{n}(\theta_U(\tilde{F}_n)-\theta(F))\overset{\mathcal{D}}{\to} N(0,4\Var(\phi(X_1))).
\end{align*}
On the other hand, we are in the degenerate regime if  $\var(\phi(X_1))=0$, where it holds that
\begin{align*}
n(\theta_U(\tilde{F})-\theta(F))&\overset{\mathcal{D}}{\to}\sum_{i=1}^\infty \lambda_i(\xi^2_i-1),
\end{align*}
where $\lambda_1,\lambda_2,\ldots$ are constants in $\R$, and $\xi_1,\xi_2,\ldots$ are independent standard Gaussian random variables. Similar results hold for V-statistics under slightly stronger assumptions. Indeed, if the extra condition $\E(|K(X_1,X_1)|)<\infty$  holds, then the exact same limit is obtained in the non-degenerate regime, while for the degenerate regime,  $n(\theta(\tilde{F}_n)-\theta(F))\overset{\mathcal{D}}{\to}\E(K(X_1,X_1))+\sum_{i=1}^\infty \lambda_i(\xi^2_i-1)$. We refer to the book of \citet{Koroljuk94}
for a comprehensive account of the theory of V- and U- statistics.

In this paper, we study the analogue of V- and U-statistics in the setting of right-censored data, that usually appears in Survival Analysis applications, in which we observe samples of the form $(X_i,\Delta_i)_{i=1}^n$ where $X_i \in (0,\infty)$ and $\Delta_i \in \{0,1\}$. Here, $\Delta_i=1$  indicates that $X_i$ is an actual sample from $F$, while $\Delta_i= 0$ indicates that $X_i$ corresponds to a right-censored observation.  Similar to the uncensored setting, we are interested on estimating $\theta(F)$, however, as the data is right-censored, it is not possible to compute the canonical V- and U-statistics as the empirical distribution $\tilde{F}_n$ is not available in this setting. Instead, we propose to replace the empirical distribution by the Kaplan-Meier estimator $\widehat{F}_n$ which is the standard estimator for $F$ in the setting of right-censored data. 

The Kaplan-Meier V- and U-statistics are defined as
\begin{align*}
\theta(\widehat{F}_n)&=\sum_{i=1}^n\sum_{j=1}^nW_iW_jK(X_{[i:n]},X_{[j:n]})
\end{align*}
and
\begin{align*}
\theta_U(\widehat{F}_n)&=\frac{\sum_{i\neq j}W_iW_jK(X_{[i:n]},X_{[j:n]})}{\sum_{i\neq j}W_i W_j},
\end{align*}
respectively, where $W_i$, are the so-called Kaplan-Meier weights and $X_{[i:n]}$ denotes the $i$-th order statistic. In this paper we conveniently write $\theta(\widehat F_n)$ as 
\begin{align*}
\theta(\widehat{F}_n)=\int_{0}^{\infty}\int_{0}^{\infty} K(x,y)d\widehat{F}_n(x)d\widehat{F}_n(y),
\end{align*}
which is known in the literature as a Kaplan-Meier double integral.

Asymptotic properties of Kaplan-Meier integrals have been studied by several authors. For the simplest case of univariate functions, Central Limit Theorems for $\int_{0}^{\infty} K(x)d\widehat{F}_n(x)$ were obtained in full generality by \citet{stute1995} and  \citet{akritas2000}. \citet{stute1995} achieved the result by expressing the Kaplan-Meier estimator as the sum of i.i.d. random variables plus some asymptotically negligible terms. By using this representation, Stute was able to deal with more general functions $K$ than preceding approaches (e.g. \citep{gill1983large,Yang1994}), however the terms in the i.i.d. representation are quite complicated, leading to strong assumptions. This problem is aggravated in the case of distributions with atoms. \citet{akritas2000} improved the result of \citet{stute1995}, obtaining weaker conditions in a more general framework. This was accomplished by using the martingale arguments developed by \citet{gill1980censoring,gill1983large}, and the identities and inequalities developed by \citet{ritov1988} and \citet{Efron1990}.

Multiple Kaplan-Meier integrals were studied by \citet{gijbels1991almost} and by \citet{bose2002asymptotic}. \citet{gijbels1991almost} studied a simplification of the  problem which considers the class of truncated Kaplan-Meier integrals $\int_0^t\ldots\int_0^tK(x_1,\ldots,x_m)\prod_{j=1}^m\widehat{F}_n(x_j)$, where $t$ is a fixed value, avoiding integration over the whole support of the observations. Then, by using an asymptotic i.i.d. sum representation of the Kaplan-Meier estimator $\widehat{F}_n$ together with integration by parts,  the authors derive an almost sure canonical V-statistic representation of $\theta(\widehat{F}_n)-\theta(F)$ up to an error of order $O(n^{-1}\log(n))$. While this result allows them to derive limiting distributions in the non-degenerate case (by scaling by $\sqrt{n}$), it is not possible to obtain results for the degenerate case since the error is too large to be scaled by $n$. Moreover, their representation is restricted to continuous distribution functions $F$.  \citet{bose2002asymptotic} analysed Kaplan-Meier double integrals in a more general setting by using a generalisation of the i.i.d. representation of the Kaplan-Meier estimator derived by \citet{stute1995} for uni-dimensional Kaplan-Meier integrals. By using this representation, \citet{bose2002asymptotic} were able to write the Kaplan-Meier double integral as a V-statistic plus some error terms. Nevertheless, similar to the univariate case, the error terms, that appear as consequence of using this approximation, are quite complicated and thus, dealing with them requires very strong and somewhat artificial conditions which are hard to verify in practice.

An alternative estimator of $\theta(F)$ in the right-censored setting is obtained by using the so-called \emph{Inverse Probability of censoring weighted} (IPCW) estimator of $F$ introduced by \citep{Satten2001TheKaplan}, which can be seen as a simplification of the Kaplan-Meier estimator. IPCW U-statistics coincide with Kaplan-Meier U-statistics when the survival times are continuous and the largest sample point is uncensored \cite{Satten2001TheKaplan}. IPCW U-statistics were studied by \citet{datta2010inverse}, however, they only provide results for the non-degenerate regime.

There are several works that study limit distributions of V- and U-statistics in the setting of dependent data (see \cite{Sen1972,Denker1986, Yoshihara1976,Dewan2002,Dehling2010,
Beutner2012,Beutner2014}). Nevertheless, most of these results are tailored for specific types of dependency and thus they are not suitable or it is not clear how to translate these results into our setting. The recent approach of \citet{Beutner2014} provides a general framework which can be applied to right-censored data, however its application is limited to very well-behaved cases. This is mainly because such an approach is based on an integration-by-parts argument, requiring the function $K$ to be locally of bounded variation, and thus denying the possibility of working with simple kernels like $K(x,y) = \ind_{\{x+y>1\}}$. Also, it is required to establish convergence of $\sqrt{n}(\widehat{F}_n - F)$ to a limit process (under an appropriate metric), leading to stronger conditions than the ones considered in our approach. Moreover, such a general result is less informative about the limit distribution than ours.
  
In this paper, we obtain limit results for Kaplan-Meier V-statistics. Our proof is based on two steps. First, we find an asymptotic canonical V-statistic representation of the Kaplan-Meier V-statistic, and second, we use such a representation to obtain limit distributions under an appropriate normalisation. We also obtain similar results for Kaplan-Meier U-statistics. Our results not only provide convergence to the limit distribution, but we also find closed-form expressions for the asymptotic mean and variance.

Applications to goodness-of-fit are provided. In particular, we study a slight modification of the Cramer-von Mises statistic under right-censoring that can be represented as a Kaplan-Meier V-statistic. Under the null hypothesis, we find its asymptotic null distribution, and we obtain closed-form expressions for the asymptotic mean and variance under a specific censoring distribution. Our results agree with those obtained by \citet{koziol1976cramer}. We also provide an application to hypothesis testing using the Maximum Mean Discrepancy (MMD), a popular
distance between probability measures frequently used in the Machine Learning community. Under the null hypothesis and assuming tractable forms for $F$ and $G$, we obtain the asymptotic limit distribution, as well as the asymptotic mean and variance of the test-statistic. 

Our results hold under conditions that are  quite reasonable, in the sense that they require integrability of terms that are very close to the variance of the limit distribution. Compared to the closest works to ours, the approach of Bose and Sen \citep{bose2002asymptotic} and the IPCW approach \citep{datta2010inverse},  our conditions are much weaker and easy to verify. We explicitly compare such conditions in Section~\ref{sec:comparisonconditions}.

\subsection{Notation}\label{sec:NotaAsu}
We establish some general notation that will be used throughout the paper. We denote $\R_+ = (0,\infty)$. Let $f:\R_+\to\R_+$ be an arbitrary right-continuous function, we define $f(x-) = \lim_{h \to 0}f(x-|h|)$ and  $(\Delta f)(x)=f(x)-f(x-)$. In this work we make use of standard  asymptotic notation \cite{janson2011probability} (e.g. $O_p$, $o_p$, $\Theta_p$, etc.) with respect to the number of sample points $n$. In order to avoid large parentheses, we write $X = O_p(1)Y$ instead of $X= O_p(Y)$, especially if the expression for $Y$ is very long. Given a sequence of stochastic processes $(W_n(x): x\in \mathcal X)$, depending on the number of observations $n$, and a function $f(x)>0$, we say that $W_n(x) = O_p(f(x))$ uniformly on $x\in A_n$, if and only if $\sup_{x\in A_n}\frac{|W_n(x)|}{f(x)} = O_p(1)$, where $A_n\subseteq \mathcal X$ is a set that may depend on $n$. 
\subsubsection{Right-censored data}

Right-censored data consists of pairs $((X_i,\Delta_i))_{i=1}^n$, where $X_i=\min(T_i,C_i)$  denotes the minimum between  survival times of interest $T_i\overset{i.i.d.}{\sim} F$ and  censoring times $C_i\overset{i.i.d.}{\sim}G$,  and  $\Delta_i=\ind_{\{T_i\leq C_i\}}$ is an indicator of whether we actually observe the survival time $T_i$ or not, that is, $\Delta_i=1$ if $T_i\leq C_i$, and $\Delta_i=0$ otherwise. We assume the survival times $T_i$'s are independent of the censoring times $C_i$'s which is known as \emph{non-informative} censoring, and it is a standard assumption in applications.

We denote by $S(x)=1-F(x)$ and by $\Lambda(x)= \int_{(0,x]} S(t-)^{-1}dF(t)$, respectively, the survival and cumulative hazard functions associated with the  survival times $T_i$. The common distribution function associated with the observed right-censored times $X_i=\min\{T_i,C_i\}$ is denoted by $H$. Note that $1-H(x)= (1-G(x))S(x)$ due to the non-informative censoring assumption. For simplicity, we assume that $F$ and $G$ are measures on $\R_+$, otherwise, we can apply an increasing transformation to the random variables, e.g. $e^{X_i}$. Notice that we do not impose any further restriction to the distribution functions, particularly, $F$ and $G$ are allowed to share discontinuity points.

\subsubsection{Kaplan-Meier estimator}
The Kaplan-Meier estimator \citep{kaplan1958nonparametric} is the non-parametric maximum likelihood estimator of $F$ in the setting of right-censored data. It is defined as $\widehat{F}_n(x)=\sum_{X_{[i:n]}\leq x} W_i$, where $W_i=\frac{\Delta_{[i:n]}}{n}\prod_{j=1}^{i-1}\left(1+\frac{1-\Delta_{[j:n]}}{n-j}\right)$ are the so-called  Kaplan-Meier weights, $X_{[i:n]}$ is the $i$-th order statistic of the sample $X_1,\ldots,X_n$, and $\Delta_{[i:n]}$ is its corresponding censor indicator.  To be very precise, ties within uncensored times or within censored times are ordered arbitrarily, while ties among uncensored and censored times are ordered such that censored times appear later.  Observe that when all the observations are uncensored, that is, when $\Delta_i=1$ for all $i$, each weight $W_i$ becomes to $1/n$ and thus $\widehat{F}_n$ becomes the empirical distribution of $T_1,\ldots,T_n$. Finally, we denote by $\widehat{S}_n(x)=1-\widehat{F}_n(x)$  the corresponding estimator of $S(x)$.

\subsubsection{Counting processes notation}\label{sec:notationCounting}
In this work we use standard Survival Analysis/Counting Processes notation. For each $i\in\{1,\ldots,n\}$ we define the individual and pooled counting processes by $N_i(t) = \ind_{\{X_i \leq t\}}\Delta_i$ and $N(t) = \sum_{i=1}^n N_i(t)$ respectively. Notice that the previous processes are indexed in $t\geq 0$. Similarly, we define the individual and pooled risk functions by $Y_i(t) = \Ind_{\{X_i \geq t\}}$ and $Y(t) = \sum_{i=1}^n Y_i(t)$, respectively. 

We assume that all our random variables are defined in a common filtrated probability space $(\Omega, \mathcal F, (\mathcal F_t)_{t\geq 0}, \Prob)$,  where $\mathcal F_t$ is generated by
\begin{align*}
\{\Ind_{\{X_i \leq s, \Delta = 1\}},\Ind_{\{X_i \leq s, \Delta = 0\}}: 0\leq s\leq t, i\in\{0,\ldots,n\}\},
\end{align*}
and the $\Prob$-null sets of $\mathcal F$. It is well-known that $N_i(t)$, $N(t)$ and the Kaplan-Meier estimator $\widehat F_n(t)$ are adapted to  $(\mathcal{F}_t)_{t\geq 0}$,  and  that $Y_i(t)$ and $Y(t)$ are  predictable processes with respect to $(\mathcal{F}_t)_{t\geq 0}$. Yet another well-known fact is that $N_i(t)$ is increasing  and its compensator is given by $\int_{(0,t]} Y_i(s)d\Lambda(s)$. We define the individual and pooled $(\mathcal F_t)$-martingales by $M_i(t) = N_i(t)-\int_{(0, t]}Y_i(s)d\Lambda(s)$ and $M(t) = \sum_{i=1}^n M_i(t)$, respectively. 

For a martingale $W$, we denote by $\langle W\rangle$  its predictable variation process, and by $[W ]$ its quadratic variation process.

Due to the simple nature of the processes that appear in this work, i.e. counting processes, checking integrability and/or square-integrability is very simple and thus we state these properties without giving an explicit proof. For more information about counting processes in survival analysis we refer to \cite{Flemming91}.

\subsubsection{Interior and Exterior regions}\label{sec:notationinteriorExterior}
Let $I_H=\{t>0: H(t-)<1\}$ be the interval in which $X_i=\min\{T_i,C_i\}$ takes values. Define $\tau=\tau_H=\sup\{t:1-H(t)>0\}$ and notice that $\tau \in I_H$ if and only if $H$ has a discontinuity at $\tau$. 

Define $\tau_n = \max\{X_1,\ldots, X_n\}$. We denote by $I=(0, \tau_n]$ the \emph{interior} region in which we observe data points, and by $E=I_H \setminus I$ the \emph{exterior} the region. Notice that both $I$ and $E$ depend on $\tau_n$ even if we do not explicitly write it.

In this work, the integral symbol $\int_a^b$ means integration over the interval $(a,b]$, unless we state otherwise. An exception to this rule is when we integrate over the interval $I_H$, in which case, instead of writing $\int_{I_H}$, we write $\int_0^{\tau}$ and define  $\int_0^{\tau}=\int_{(0,\tau]}$ if $\tau\in I_H$ and $\int_0^{\tau}=\int_{(0,\tau)}$ if $\tau\not\in I_H$. 

Lastly, let $g:\R_+\to\R$ be an arbitrary function, then $\int_{\R_+}g(x)d\widehat{F}_n(x)=\int_{0}^\tau g(x)d\widehat{F}_n(x)=\int_{0}^{\tau_n}g(x)d\widehat{F}_n(x)$ as $\widehat{F}_n(t)=\widehat{F}_n(\tau_n)$ for all $t>\tau_n$. The same holds when integrating with respect the martingale $M(t)$. 

\subsubsection{Efron and Johnstone's Forward Operator} 

We consider the forward  operator $A:L^2(F)\to L^2(F)$, independently introduced by \citet{ritov1988} and \citet{Efron1990}, defined by
\begin{align}
(Ag)(x) &= g(x)-\frac{1}{S(x)}\int_{x}^{\infty} g(s)dF(s).\nonumber
\end{align}
 For  a bivariate function $g:\R_+^2\to\R$ such that $g\in L^2(F \times F)$, we denote by $A_i$, the operator $A$ applied only to the $i$-th coordinate of the function $g$, e.g.  $(A_2g)(x,y) = g(x,y)- \frac{1}{S(y)}\int_{y}^{\tau}g(x,s)dF(s)$. Note that $A_1$ and $A_2$ commute.

Similarly, for $g:\R_+\to\R$ such that $g\in L^2(F)$, we define $\widehat A:L^2(F)\to L^2(F)$, as
\begin{align}
(\widehat A g)(x) =  \begin{cases} 
      g(x)-\frac{1}{S(x)}\int_x^{\tau_n}g(s) dF(s), & 0<x\leq \tau_n\\
      0, & x>\tau_n
   \end{cases}.\nonumber
\end{align}
Observe that the difference between $\widehat A$ and the forward operator A is the upper limit of integration. For bivariate functions, we define $\widehat A_i$ as the operator $\widehat A$ applied only to the $i$-th coordinate.

Notice that for $x \leq \tau_n$,
\begin{align}
((\widehat{A}-A)g)(x)&=\frac{1}{S(x)}\left(\int_x^\infty g(s)dF(s)-\int_x^{\tau_n} g(s)dF(s)\right)\nonumber\\
&=\frac{1}{S(x)}\int_{\tau_n}^\infty g(s)dF(s).\label{eqn:diffAes}
\end{align}
\sloppy{Also, notice that if $g(x)=1$ is a constant function, then $(Ag)(x)=1-\frac{1}{S(x)}\int_x^\infty 1dF(s)=0$, and $(\widehat{A}g)(x)=1-\frac{S(x)-S(\tau_n)}{S(x)}=\frac{S(\tau_n)}{S(x)}$. Finally, observe that the definitions of $A$ and $\widehat A$ depend only on $F$ and it does not consider the censoring distribution $G$.}

\section{Main Results}
The  Kaplan-Meier V-statistic associated with $\theta(F)$ is defined by
\begin{align*}
\theta(\widehat{F}_n)&= \int_{0}^{\tau_n}\int_{0}^{\tau_n} K(x,y) d\widehat{F}_n(x) d\widehat{F}_n(y)=\sum_{i=1}^n \sum_{j=1}^n W_i W_j K(X_{[i:n]}, X_{[j:n]}),
\end{align*}
where the second equality follows from the definition of the Kaplan-Meier estimator $\widehat{F}_n$.

 \citet{bose1999strong} proved that 
 \begin{align*}
 \theta(\widehat{F}_n)\overset{a.s.}{\to} \theta(F;\tau)= \int_{0}^{\tau}\int_{0}^{\tau} K(x,y)dF(x)dF(y),
 \end{align*}
as $n$ approaches infinity. Notice that the limit  $\theta(F;\tau)$ has a dependency on $\tau$ since the data is right-censored, and thus we do not observe any survival time beyond the time $\tau$.  Consider the difference $\Inte-\theta(F;\tau)$, and notice it can be decomposed into two error terms:
\begin{align}
\Inte-\theta(F;\tau) &= \int_{0}^{\tau}\int_{0}^{\tau} K(x,y) (d\widehat{F}_n(x)d\widehat{F}_n(y)-dF(x)dF(y))\nonumber\\
&= 2\alpha(F,\widehat{F}_n)+\beta(F,\widehat{F}_n),\label{eqn:expansionI-T}
\end{align}
where 
\begin{align}
\alpha(F,\widehat{F}_n)= \int_{0}^{\tau}\phi(y)d(\widehat{F}_n-F)(y),\nonumber
\end{align}
where the projection $\phi$ is given by
\begin{align}
\phi(y)=\int_{0}^{\tau} K(x,y)dF(x),\label{eqn:defiphi}
\end{align}
and
\begin{align}
\beta(F,\widehat{F}_n)&=\int_{0}^{\tau}\int_{0}^{\tau} K(x,y)d(\widehat{F}_n-F)(x)d(\widehat{F}_n-F)(y).\label{eqn:defiBeta}
\end{align}
Note that in our setting, our definition of $\phi$ integrates up to $\tau$, instead of the whole support of $F$ as in Equation~\eqref{eqn:random1294hsd9} used in the uncensored setting. To ease notation, we write $\alpha$ and $\beta$ instead of $\alpha(F,\widehat{F}_n)$ and $\beta(F,\widehat{F}_n)$, respectively. Each error term $\alpha$ and $\beta$ can be seen as a first and second order approximation of the difference $\Inte - \theta(F;\tau)$. That being said, we expect that the error term $\alpha$ is of a much larger order than $\beta$. Indeed, it holds that $\alpha = O_p(n^{-1/2})$ and $\beta = O_p(n^{-1})$. This suggests the use of two different scaling factors, splitting our main result into two cases: the \emph{non-degenerate} and the \emph{degenerate} case. In the first case, we will show that, under appropriate conditions, $\sqrt{n}(\Inte-\theta(F;\tau))$ converges in distribution to a zero-mean normal random variable when $n$ approaches infinity. This result will follow from proving a normal limit distribution for the scaled  error term $\sqrt{n}\alpha$ and an in-probability convergence to zero of the scaled error $\sqrt{n}\beta$. In the degenerate case, the error term $\alpha$ is trivially 0, and thus we only care about the term $\beta$. We will show that $n\beta$ converges in distribution to a linear combination of (potentially infinity) independent $\chi^2$ random variables plus a constant. From these results, we will be able to derive analogue results to those in the canonical V-statistics setting.

To express our results and conditions, we define the kernel $K':\R_+^2\to\R$ as $K'(x,y)=(A_1A_2K)(x,y)$, which, by the definition of the operators $A_1$ and $A_2$, is equal to
\begin{align}\label{eqn:defK'}
K(x,y)-\int_{x}^{\tau}K(s,y)\frac{dF(s)}{S(x)}-\int_{y}^{\tau}K(x,t)\frac{dF(t)}{S(y)}+\int_{x}^{\tau}\int_{y}^{\tau}K(s,t)
\frac{dF(s)}{S(x)}\frac{dF(t)}{S(y)}.
\end{align}
We introduce two sets of conditions, one for the non-degenerate case,  and  one for the degenerate case.
\begin{condition}[non-degenerate case: scaling factor $\sqrt{n}$]\label{assu:nondegenerate}
Assume the following conditions hold:
\begin{enumerate}
\item[i)] $\int_{0}^{\tau} \int_{0}^{\tau}\frac{ K(x,y)^2 }{(1-G(x-))}dF(y)dF(x)<\infty$,
\item[ii)] $\int_{0}^{\tau} \left(\frac{S(x-)}{1-G(x-)}\right)^{1/2}|K'(x,x)|dF(x)<\infty$.
\end{enumerate}
\end{condition}

\begin{condition}[degenerate case: scaling factor $n$]\label{assu:degenerate}
Assume the following conditions hold:
\begin{enumerate}
\item[i)] $\int_{0}^{\tau}\int_{0}^{\tau} \frac{K(x,y)^2}{(1-G(x-))(1-G(y-))}dF(x)dF(y)<\infty$,
\item[ii)] $\int_{0}^{\tau} \frac{|K'(x,x)|S(x)}{1-H(x-)}dF(x)<\infty$.
\end{enumerate}
\end{condition}
Notice that Condition~\ref{assu:degenerate} implies Condition~\ref{assu:nondegenerate}. 

\subsection{Results for Kaplan-Meier V-statistics}
\textbf{i) The non-degenerate case:  $\sqrt{n}$-scaling:}
Equation \eqref{eqn:expansionI-T} states $\sqrt{n}(\theta(\widehat{F}_n)-\theta(F;\tau))=2\sqrt{n}\alpha+\sqrt{n}\beta$. Recall that $\alpha$ is defined in terms of the projection $\phi$ defined in Equation~\eqref{eqn:defiphi}. Then, the main result follows under Condition~\ref{assu:nondegenerate} from a standard application of the Central Limit Theorem (CLT) derived by \citet{akritas2000} for univariate Kaplan-Meier integrals and by proving $\sqrt{n}\beta\overset{\Prob}{\to}0$. Akritas proved
\begin{align*}
\sqrt{n}\alpha\overset{\mathcal{D}}{\to}N(0,\sigma^2),
\end{align*}
where 
\begin{align}\label{eqn:varianceNondege}
\sigma^2= \int_{0}^\tau \frac{S(x)}{1-H(x-)}(A \phi)^2(x)dF(x).
\end{align}
As noticed by \citet{Efron1990}, $\sigma^2$ is finite if $\int_0^\tau \frac{\phi(x)^2}{1-G(x-)}dF(x)<\infty$, which is implied by Condition \ref{assu:nondegenerate}. 

\begin{theorem}\label{thm:convergenceNondege} Under Condition~\ref{assu:nondegenerate}, it holds
\begin{align}
\sqrt{n}\alpha \overset{\mathcal D}{\to} N(0,\sigma^2), \quad \text{and} \quad \sqrt{n}\beta \overset{\Prob}{\to} 0,\nonumber
\end{align}
and thus
\begin{align}
\sqrt{n}(\Inte-\theta(F))\overset{\mathcal D}{\to} N(0,4\sigma^2),\nonumber
\end{align}
where $\sigma<\infty$ is given in Equation~\eqref{eqn:varianceNondege}.
\end{theorem}

\textbf{ii) The degenerate case: $n$-scaling.} The previous result considers that $\sigma^2>0$. Notice this is not satisfied if $\alpha=0$ as it implies $\sigma^2=0$. In turn, $\alpha=0$ can be deduced from either of the following conditions of the projection $\phi$ defined in Equation~\eqref{eqn:defiphi}: i) $\phi(x)=0$, $F$-a.s., or ii) $\phi(x)=c$, $F$-a.s., for some non-zero constant $c$  and $\widehat{S}_n(\tau)=S(\tau)$ a.s. for all  $n$ large enough. In the theory of V- and U-statistics, these conditions are known as the \textit{degeneracy properties}. In such a case, the $\sqrt{n}$-scaling does not capture the nature of the asymptotic distribution of $\Inte-\theta(F;\tau)$, suggesting that we need to consider a larger scaling factor. 

Recall $\alpha=\int_0^\tau \phi(x)(d\widehat{F}_n-dF)(x)$. Then, it is straightforward to verify, that i) $\phi(x)=0$, $F$-a.s., implies $\alpha=0$, and that ii) if $\phi(x)=c$, $F$-a.s., with $c \neq 0$, then $\alpha=c(\widehat{F}_n(\tau)-F(\tau))=0$ if and only if $\widehat{F}_n(\tau)=F(\tau)$ a.s. for all large $n$ (notice this condition is trivially satisfied in the uncensored case). In those cases the information of the limit distribution is contained in the term $\beta$.

Define $J:(I_H\times\{0,1\})^2\to\R$
\begin{align}
J((x,r),(x',r'))&= \int_0^{\tau}\int_0^{\tau} \frac{K'(s,t)}{(1-G(s-))(1-G(t-))}dm_{x,r}(s)dm_{x',r'}(t),\label{eqn:Jfunction}
\end{align}
where $dm_{x,r}(s)=r\delta_x(s)-\ind_{\{x\geq s\}}d\Lambda(s)$, and notice that $dm_{X_i,\Delta_i}=dM_i$, and recall that $M_i$ is the $i$-th individual martingale defined in Section~\ref{sec:notationCounting}.

\begin{theorem}\label{thm:decompositionDege}
Under Condition~\ref{assu:degenerate}, it holds that
\begin{align}
n\beta \overset{\mathcal D}{\to} \int_{0}^{\tau}\frac{S(x)}{1-H(x-)}K'(x,x)dF(x)+ \Psi\nonumber
\end{align}
where $\Psi = \sum_{i=1}^{\infty} \lambda_i (\xi_i^2-1)$, $\xi_i\overset{i.i.d}{\sim}N(0,1)$ and the  $\lambda_i$'s are the eigenvalues associated with the integral operator $T^J:L^2(X,\Delta)\to L^2(X,\Delta)$, defined as
\begin{align*}
(T^Jf)(X_1,\Delta_1)=\E\left(J((X_1,\Delta_1),(X_2,\Delta_2))f(X_2,\Delta_2)\given (X_1,\Delta_1)\right),
\end{align*} 
where $L^2(X,\Delta)$ is the space of square-integrable functions with respect the measure induced by $(X_1,\Delta_1)$.

Moreover, $\E(\Psi) = 0$ and \begin{align}
\E(\Psi^2) = 2\int_{0}^{\tau}\int_{0}^{\tau}
\frac{K'(x,y)^2S(x)S(y)}{(1-H(x-))(1-H(y-))}dF(x)dF(y).\nonumber
\end{align}
\end{theorem}

An immediate consequence of the  previous Theorem is the asymptotic behaviour of the degenerate case for the Kaplan-Meier V-statistic.
\begin{corollary}\label{Corollary:Vstatsdege}
Suppose one of the following \textit{degeneracy conditions} hold.  
\begin{enumerate}
\item[i)] $\phi(x)=0$, $F$-a.s.,
\item[ii)] $\phi(x)=c$, $F$-a.s. for some non-zero constant $c$, and exists $N>0$ such that $\widehat{S}_n(\tau) = S(\tau)$ a.s. for all  $n \geq N$.
\end{enumerate}
Then, under Condition~\ref{assu:degenerate},
\begin{align}
n(\Inte-\theta(F;\tau)) \overset{\mathcal{D}}{\to} \int_{0}^{\tau}\frac{S(x)}{1-H(x-)}K'(x,x) dF(x)+ \Psi\nonumber
\end{align}
where $\Psi = \sum_{i=1}^{\infty} \lambda_i (\xi_i^2-1)$ is defined as in Theorem~\ref{thm:decompositionDege}.
\end{corollary}

As a part of the proof, we find an asymptotic representation of $\beta$ as a canonical V- statistic, this representation is as following.
\begin{theorem}\label{lemma:afterMartingalenb}
Under Condition~\ref{assu:degenerate} it holds that
\begin{align*}
\beta = \frac{1}{n^2}\sum_{i=1}^n\sum_{j=1}^n J((X_i,\Delta_i),(X_j,\Delta_j))+o_p(n^{-1}).
\end{align*}
\end{theorem}

All the results of this section are proved in Section~\ref{sec:Road Map} except for Theorem~\ref{lemma:afterMartingalenb}, which is proved in Section~\ref{sec:Interior}.

\subsection{Kaplan-Meier U-statistics}
The Kaplan-Meier U-statistic is defined by
\begin{align*}
\theta_U(\widehat{F}_n)&=\frac{\sum_{i\neq j} W_i W_j K(X_{[i:n]}, X_{[j:n]})}{\sum_{i\neq j} W_i W_j}=\frac{\Inte-\sum_{i=1}^nK(X_i,X_i)W_i^2}{\sum_{i\neq j} W_i W_j},
\end{align*} 
where the second equality follows from adding and subtracting the diagonal term $(\sum_{i=1}^nK(X_i,X_i)W_i^2)/(\sum_{i\neq j} W_i W_j)$.

Without loss of generality, assume $\theta(F;\tau)=0$. Then, the asymptotic distribution of $\theta_U(\widehat{F}_n)$ can be related to the one for $\Inte$ by analysing the asymptotic behaviour of  $\sum_{i\neq j}W_iW_j$ and $\sum_{i=1}^nK(X_i,X_i)W_i^2$. For the first term, \citet{bose1999strong} proved that $\sum_{i\neq j}W_iW_j\overset{a.s.}{\to}F(\tau)^2$. For the second term we enunciate the following result, which is proved in \Appendice~\ref{Appendix:SecDiag}. 
\begin{lemma}\label{lemma:diagKnonDege}
If $\int_0^{\tau} |K(x,x)|\sqrt{\frac{S(x-)}{1-G(x-)}}dF(x)<\infty$, then
\begin{align*}
\sqrt{n}\sum_{i=1}^n K(X_i,X_i)W_i^2&=o_p(1).
\end{align*}
Additionally, if $\int_0^{\tau} \frac{|K(x,x)|}{1-G(x-)}dF(x)<\infty$, then
\begin{align}
n\sum_{i=1}^n K(X_i,X_i)W_i^2 &= \int_0^\tau \frac{K(x,x)}{1-G(x-)}dF(x)+o_p(1).\nonumber
\end{align}
\end{lemma}
The previous Lemma combined with the results obtained in the previous section allow us to deduce the following results for Kaplan-Meier U-statistics.
\begin{corollary}
Assume Condition \ref{assu:nondegenerate}, and additionally assume that we have $\int_0^{\tau} |K(x,x)|\sqrt{\frac{S(x-)}{1-G(x-)}}dF(x)<\infty$. Then, it holds that
\begin{align}
\sqrt{n}\theta_U(\widehat{F}_n)\overset{\mathcal{D}}{\to}N\left(0,4\sigma^2/F(\tau)^4\right).\nonumber
\end{align}
\end{corollary}

\begin{corollary}
Suppose one of the following \it{degeneracy conditions} hold.  
\begin{enumerate}
\item[i)] $\phi(x)=0$, $F$-a.s.,
\item[ii)] $\phi(x)=c$, $F$-a.s. for some non-zero constant $c$, and exists $N>0$ such that $\widehat{S}_n(\tau) = S(\tau)$ a.s. for all  $n \geq N$.
\end{enumerate}
Assume Condition \ref{assu:degenerate} and that $\int_{0}^{\tau} \frac{|K(x,x)|}{1-G(x-)}dF(x)<\infty$. Then, it holds that
\begin{align*}
n\theta_U(\widehat{F}_n)\overset{\mathcal{D}}{\to}\frac{1}{F(\tau)^2}\left(\int_0^\tau\left(\frac{S(x)K'(x,x)}{1-H(x-)}-\frac{K(x,x)}{1-G(x-)}\right)dF(x)+\Psi\right),
\end{align*}
where $\Psi$ is as in Theorem~\ref{thm:decompositionDege}.
\end{corollary}
\subsection{Analysis of Conditions~\ref{assu:nondegenerate} and \ref{assu:degenerate}, and comparison with related works}\label{sec:comparisonconditions}
In this section we discuss  our conditions, and  we compare them with the work of \citet{bose2002asymptotic} and \citet{datta2010inverse}. 

We begin by analysing Condition \ref{assu:nondegenerate}, used in the non-degenerate regime, which implies Theorem \ref{thm:convergenceNondege}. Recall that \citet{Efron1990} showed that the variance of the limit distribution in Theorem~\ref{thm:convergenceNondege} is finite if 
\begin{align}
\int_0^\tau \frac{\left(\phi(x)\right)^2}{1-G(x-)}dF(x)=\int_0^\tau \frac{1}{1-G(x-)}\left(\int_0^\tau K(x,y)dF(y)\right)^2dF(x)<\infty, \label{eqn:randomv8fb8h2}
\end{align} 
whereas our Condition~\ref{assu:nondegenerate}.i requires 
\begin{align*}
\int_0^\tau\int_0^{\tau} \frac{K(x,y)^2}{1-G(x-)}dF(y)dF(x)<\infty,
\end{align*}
which is very close to term in Equation~\eqref{eqn:randomv8fb8h2}. Indeed, there is just one  \emph{Cauchy-Schwarz} inequality gap from the condition of \citet{Efron1990}, suggesting little room for improvement. On the other hand, Condition~\ref{assu:nondegenerate}.ii is a standard condition to deal with the diagonal term that appears in the V-statistic representation. It is only used in Lemma~\ref{lemma:sqrtnQD} and it is usually much simpler to verify due to the multiplicative factor $\sqrt{S(x-)}$ that appears in the integral, which makes the tail much lighter.

We compare our Condition~\ref{assu:nondegenerate} with  the conditions of
Theorem 1 of \citet{bose2002asymptotic}, which establishes the same limit result as our Theorem~\ref{thm:convergenceNondege} under different conditions. Theorem 1 of \cite{bose2002asymptotic} requires our Condition \ref{assu:degenerate}.i (which implies our Condition \ref{assu:nondegenerate}.i), together with three extra conditions involving the function
\begin{align}\label{eqn:defiCfunction}
C(x) = \int_0^{x-} \frac{dG(s)}{S(s)((1-G(s))^2}.
\end{align}
For example, one  of the extra conditions required is  
\begin{align}\label{eqn:condBoseSen}
\int_0^{\tau}\int_0^{\tau} \frac{|K(x,y)|C(x)C(y)}{(1-G(x-))(1-G(y-))}dF(x)dF(y)< \infty,
\end{align}
which, compared to our Condition~\ref{assu:nondegenerate}.i, is much harder to satisfy as the function $C(x)$ grows  much faster than $(1-G(x-))^{-1}$ when $x$ approaches infinity. Indeed, by assuming that $G$ and $F$ are continuous distributions,  it is not hard to verify that $(\frac{\partial}{\partial x} C(x))/(\frac{\partial}{\partial x}(1-G(x))^{-1})=1/S(x)$. Therefore, unless the kernel $K(x,y)$ decays very fast, Equation~\eqref{eqn:condBoseSen} is very hard to satisfy. In example~\ref{example:exampleconditions} below, we show that $C(x)$ can grow exponentially faster than $(1-G(x))^{-1}$.

We continue by comparing our Condition~\ref{assu:nondegenerate} with the ones of
Theorem 1 of \citet{datta2010inverse}.  Theorem 1 of \citep{datta2010inverse} requires  $\int_0^\tau \frac{\phi(x)^2}{1-G(x-)}dF(x)<\infty$, which is the condition of \citet{Efron1990} for finiteness of the  variance. However, it also requires  
\begin{align*}
\int_{0}^{\tau} \left(\frac{(\phi-A\phi)(x)}{1-G(x-)} \right)^2\frac{dG(x)}{1-G(x-)} = \int_0^{\tau} \frac{S(x)}{1-G(x)}(\phi-A\phi)(x)^2 dC(x)<\infty,
\end{align*}
which is very hard to satisfy as it involves the function $C(x)$ defined in Equation~\eqref{eqn:defiCfunction}.

\begin{example}\label{example:exampleconditions}
Let $K(x,y) = xy$,  $S(x) = e^{-x}$ and $1-G(x) = S(x)^a$ with $a > 0$. Note that $\phi(x) = x$, $(A\phi)(x) = -1$ and $K'(x,y )=(A_1A_2K)(x,y) = 1$. From Theorem~\ref{thm:convergenceNondege}, we have that $\sqrt{n}(\theta(\widehat F_n)-\theta(F)) \overset{\mathcal D} {\to}N(0,4\sigma^2)$, where $\sigma=\int_0^{\infty} e^{-x+ax}dx$ (see Equation~\eqref{eqn:varianceNondege}) is finite if and only if $a<1$. 

In this setting, our Conditions \ref{assu:nondegenerate}.i and \ref{assu:nondegenerate}.ii are 
\begin{align*}
\int_0^{\infty} \int_0^{\infty} x^2y^2 e^{-x+ax}e^{-y}dxdy<\infty,\quad\text{and}\quad \int_0^{\infty} e^{-(3/2-a/2)x}<\infty,
\end{align*}
respectively, which are satisfied for $a<1$. Hence, our conditions are the best possible in this case, as the variance $\sigma^2$ of the limit distribution is finite if and only if $a<1$.

Bose and Sen's approach \cite{bose2002asymptotic} requires the finiteness of the expression in Equation~\eqref{eqn:condBoseSen}. In this example, $C(x)=a\frac{e^{x+ax}-1}{1+a}$, then Equation~\eqref{eqn:condBoseSen} is equal to
\begin{align*}
\int_0^{\tau}\int_0^{\tau} \frac{|K(x,y)|C(x)C(y)}{(1-G(x-))(1-G(y-))}dF(x)dF(y)=\left(\int_0^{\infty}x\frac{a(e^{x+ax}-1)}{(1+a)e^{-ax}}e^{-x}dx\right)^2,
\end{align*}
which is infinite for all $a>0$. We deduce that Theorem 1 of \cite{bose2002asymptotic} cannot be applied in this setting.   

IPCW's  approach \cite{datta2010inverse} requires  
\begin{align*}
\int_0^{\infty} x^2 e^{ax}e^{-x}dx<\infty,\quad\text{and}\quad\int_0^{\infty} a(x+1)^2e^{2ax}dx<\infty.
\end{align*}

While the first equation is satisfied for $a<1$, the second equation cannot be satisfied for any $a>0$. Hence, Theorem 1 of \cite{datta2010inverse}  does not hold in this setting.

In the previous example note that $C(x)=a\frac{e^{x+ax}-1}{1+a}$ grows exponentially fast with $x$. Therefore, to use the respective theorems of \cite{bose2002asymptotic} and \cite{datta2010inverse} we will need to use kernels that decay exponentially fast.

\end{example}

We continue by analysing Condition~\ref{assu:degenerate} which is used in the degenerate case.  Observe that the integral of Condition \ref{assu:degenerate}.ii is equal to the first moment of the limit distribution of Theorem~\ref{thm:decompositionDege}, thus this condition cannot be avoided.  The variance of the limit distribution in Theorem~\ref{thm:decompositionDege}, is given by 
\begin{align}\label{eqn:randomvhjrh421d}
\int_0^{\tau}\frac{K'(x,y)^2S(x)S(y)}{(1-H(x-))(1-H(y-))}dF(x)dF(y),
\end{align}
while our Condition \ref{assu:degenerate}.i requires 
\begin{align*}
\int_0^{\tau}\frac{K(x,y)^2}{(1-G(x-))(1-G(y-))}dF(x)dF(y)<\infty,
\end{align*}
which ensures the finiteness of \eqref{eqn:randomvhjrh421d}. 
Recall that $K'$, defined in Equation~\eqref{eqn:defK'}, is given by
\begin{align*}
K(x,y)-\int_{x}^{\tau}K(s,y)\frac{dF(s)}{S(x)}-\int_{y}^{\tau}K(x,t)\frac{dF(t)}{S(y)}+\int_{x}^{\tau}\int_{y}^{\tau}K(s,t)
\frac{dF(s)}{S(x)}\frac{dF(t)}{S(y)}.
\end{align*}

From here, if we consider continuous distributions, we observe that the expression in our condition is similar to the variance given in \eqref{eqn:randomvhjrh421d}. If we consider an appropriate kernel $K$, it may happen that some terms in $K'$ cancel each other, resulting in a kernel $K'$ of much smaller  order than $K$. An example of this is the kernel $K(x,y) = (x-c)(y-c)$, where $c = \int x dF(x)$ and $S(x) = e^{-x}$. Note this kernel is similar to the previous example, but we subtract $c$ to make it degenerate. In this setting we have $K'(x,y) = 1$, hence it is easier to have finite variance than to satisfy our condition. However, in general cases we do not expect to have cancellation between the terms in $K'$ and thus $K$ and $K'$ should be of  similar order, making our Condition \ref{assu:degenerate}.i sufficient and necessary.

Up to the best of our knowledge, the work of \citet{bose2002asymptotic} is the only one that establishes results for the degenerate case in a general setting. Compared to their result, our conditions are better since their Theorem 2 has the same requirements as their Theorem 1, i.e. conditions involving the function $C(x)$, including Equation~\eqref{eqn:condBoseSen} which, as we saw in our previous example, is very hard to satisfy. Indeed, if we repeat Example~\ref{example:exampleconditions} with the kernel $K(x,y) = (x-c)(y-c)$, the conditions of Theorem~\ref{thm:decompositionDege} are satisfied for $a<1$ (in which case the asymptotic variance is well-defined), while the conditions of Bose and Sen are not satisfied.

\section{Applications}
We give two examples of applications that motivated us to study Kaplan-Meier V-statistics. First we analyse a slight variation of the Cramer-Von Mises statistic that allows us to treat it as a Kaplan-Meier V-statistic. In our second application, we measure goodness-of-fit via the Maximum Mean Discrepancy (MMD), a popular distance between probability measures frequently used in the Machine Learning community.

\begin{example}[Cram\'er-von Mises test-statistic]
Consider the problem of testing the null hypothesis $H_0:F=F_0$ against the general alternative $H_1:F\neq F_0$. The Cram\'er-von Mises statistic measures the closeness between $F$ and $F_0$ by computing
\begin{align}\label{eqn:random3913}
\int_{\R_+}(F(x)-F_0(x))^2dF_0(x).
\end{align}
When $F$ is a probability distribution function, it can be verified that Equation \eqref{eqn:random3913} equals to
\begin{align}
\theta(F)=\int_{\R_+}\int_{\R_+} K(x_1,x_2)dF(x_2)dF(x_1),\label{eqn:randomCVM123sf}
\end{align}
where  
\begin{align}
K(x,y)&=\int_{\R_+}(\ind\{x\leq t\}-F_0(t))(\ind\{y\leq t\}-F_0(t)) dF_0(t).\nonumber
\end{align}
Under the null hypothesis $H_0: F=F_0$, we estimate $\theta(F)=0$ by using  Equation \eqref{eqn:randomCVM123sf}, replacing $F$ by the Kaplan-Meier estimator $\widehat{F}_n$. Then, our test-statistic is
\begin{align}
\theta(\widehat{F}_n)=\sum_{i=1}^n\sum_{j=1}^nW_iW_jK(X_{[i:n]},X_{[j:n]}).\label{eqn:CVM test-stat}
\end{align}
Notice that the equality between Equations~\eqref{eqn:random3913} and \eqref{eqn:randomCVM123sf} is only valid when $F$ is a probability distribution, unfortunately, the Kaplan-Meier estimator $\widehat{F}_n$ is not always a probability distribution, indeed, $\widehat{F}_n$ is a probability distribution if and only if the largest observation is uncensored, thus $\theta(\widehat{F}_n)$ is slightly different from the Cram\'er-von Mises test-statistic  $\int_{0}^{\infty} (\widehat {F}_n(x)-F_0(x))^2 dF_0(x)$.

Under the null hypothesis, we observe two different asymptotic behaviours of our test-statistic $\theta(\widehat{F}_n)$, one for $F_0(\tau)<1$ and the other for $F_0(\tau)=1$. To see this, for $x\in I_H$, consider the projection $\phi$ defined in Equation~\eqref{eqn:defiphi}, which in this case is given by
\begin{align}
\phi(x)&=\int_{0}^\tau K(x,y)dF_0(y)\nonumber\\
&=\int_{0}^{\infty}\left(F_0(t\wedge \tau)-F_0(t)F_0(\tau)\right)(\ind\{x\leq t\}-F_0(t))dF_0(t),\nonumber
\end{align}
and notice that if $F_0(\tau)<1$, then $\phi(x)$ does not satisfy the {degeneracy condition} of Corollary \ref{Corollary:Vstatsdege}. Thus, by Theorem \ref{thm:convergenceNondege}, it holds that $\sqrt{n}(\theta(\widehat{F}_n)-\theta(F_0,\tau))$ is asymptotically normally distributed. On the other hand, if $F_0(\tau)=1$, then $\phi(x)$ satisfies the {degeneracy condition} of Corollary \ref{Corollary:Vstatsdege}, indeed, we have that $\phi(x)=0$ for all $x \in I_H$. Hence, under Condition \ref{assu:degenerate}, Corollary \ref{Corollary:Vstatsdege} applies, concluding that $n\theta(\widehat{F}_n)$ is asymptotically distributed as the weighted sum of i.i.d. $\chi_1^2$ random variables plus some constant term. 

For comparison purposes, we consider the alternative formulation of the Cramer-von Mises statistic by \citet{koziol1976cramer}. They consider the random integral $\Phi_n=\int_0^{\infty}(\bar{F}_n(t)-F_0(t))^2dF_0(t)$, where $\bar{F}_n$ is exactly as the Kaplan-Meier estimator, but they force $\bar {F}_n(\tau_n) = 1$ even if the largest observation is censored. For simplicity of the analysis, \citet{koziol1976cramer} assumed that the censoring distribution satisfies  $1-G(t)=S_0(t)^\gamma$ for $\gamma<2$ and  that $F_0$ is a continuous distribution. Then, based on Gaussian processes arguments, they proved that $n\Phi_n\overset{\mathcal{D}}{\to}\Phi$ where $\Phi$ denotes (a potentially infinite) linear combination of $\chi_1^2-1$ independent random variables, and that
\begin{align}
\E(\Phi)=\frac{1}{3(2-\gamma)}, \quad \text{and} \quad\Var(\Phi)=\frac{2}{9(5-\gamma)(2-\gamma)}.\nonumber
\end{align} 
Using our techniques,  we consider $\theta(\widehat{F}_n)$ as in Equation \eqref{eqn:CVM test-stat}. In this case, we get 
$K(x,y)=S_0(\max\{x,y\})+\frac{F_0(x)^2+F_0(y)^2}{2}-\frac{2}{3}$ and $K'(x,y)=\frac{S_0(\max\{x,y\})^3}{3S_0(x)S_0(y)}$ (recall the definition of $K'$ in Equation~\eqref{eqn:defK'}). By choosing $1-G(t)=S_0(t)^\gamma$, it holds $F_0(\tau)=1$ which satisfies the {degeneracy condition}, as $\phi(x)=\int_0^{\tau}K(x,s)dF_0(s) = 0$. Then, if $\gamma<1$, the conditions of Corollary~\ref{Corollary:Vstatsdege} are satisfied and thus
\begin{align*}
n\theta(\widehat{F}_n)\overset{\mathcal{D}}{\to}\int_{0}^{\infty} \frac{S_0(x)}{3S_0(x)^{\gamma}}dF_0(x)+\Psi,
\end{align*}
where $\Psi$ is as in Theorem~\ref{thm:decompositionDege}. 
Recall that $\E(\Psi)=0$, then the asymptotic mean is given by $\frac{1}{3(2-\gamma)}$ and the asymptotic variance is given by
\begin{align}
2\sum_{k=1}^\infty \lambda_k^2=2\int_0^\tau\int_0^\tau \frac{K'(x,y)^2}{S_0(x)^\gamma S_0(y)^\gamma}dF_0(y)dF_0(x)=\frac{2}{9(5-\gamma)(2-\gamma)}.\nonumber
\end{align}
Our result suggests that our estimator and the one considered by \citet{koziol1976cramer} have similar behaviours, even when rescaled by $n$. In Figure \ref{Fig1}, we show simulations of the empirical distribution of $n\theta(\widehat{F}_n)$ for different sample sizes $n$, and $\gamma\in\{0.5,1,1.5\}$. For $\gamma=0.5$ we can observe a clear convergence of the distribution functions as predicted by our results. The plot for $\gamma=1.5$ shows a shift of the distribution functions as the sample size increases, suggesting divergence. The simulations for $\gamma = 1$ are, unfortunately, not very revealing. 
\begin{figure}
        \includegraphics[width=\textwidth]{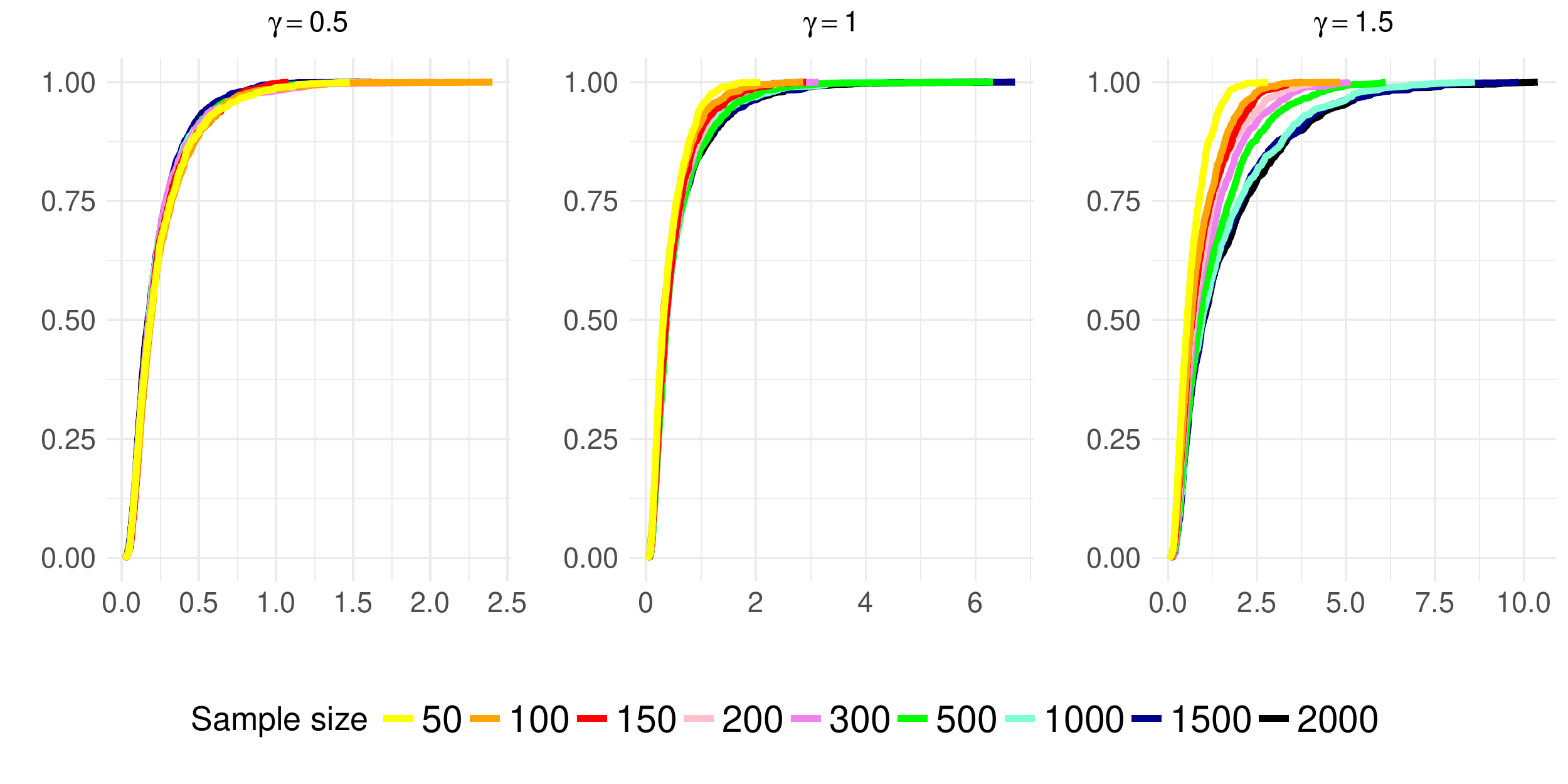}\caption{Empirical distribution functions of $n\theta(\widehat{F}_n)$ for different sample sizes and values of $\gamma$. Each empirical distribution was computed using 1000 independent realisations.}\label{Fig1}
\end{figure}
\end{example}

\begin{example}[Maximum mean discrepancy]Let $(\mathcal{H},\langle\cdot,\cdot\rangle_{\mathcal{H}})$ be a reproducing kernel Hilbert space of real-valued functions with reproducing kernel denoted by $K$. Denote by $\mathcal{P}$ the set of all probability distribution functions on $\R_+$, we define the map $\mu_\cdot:\mathcal{P}\to \mathcal H$ by $\mu_F(\cdot) = \E_{X\sim F}(K(\cdot,X))$ for any distribution function $F\in\mathcal{P}$. A reproducing kernel is called characteristic if the map $\mu$ is injective \citep{sriperumbudur2010hilbert}. It is worth mentioning that most of the standard positive-definite kernels (e.g. Gaussian and Ornstein-Uhlenbeck) are characteristic. In such a case, the map $\mu$ allows us to establish a proper distance between probability measures in terms of the norm of the space $\mathcal{H}$. That is, given two probability distributions $F$ and $F_0$, we define their distance by
\begin{align}
\|\mu_{F}-\mu_{F_0}\|_{\mathcal{H}} =\sqrt{\int_{\R_+}\int_{\R_+}K(x,y)(dF(x)-dF_0(x))(dF(y)-dF_0(y))}.\label{eqn:mmddistanceran1231}
\end{align}
Also, under the conditions stated above, such distance coincides with the Maximum mean discrepancy with respect to the unit ball of $\mathcal H$, which is defined as follows
\begin{align}
MMD_{\mathcal H}(F,F_0)=\sup_{f\in\mathcal{H},\|f\|_{\mathcal H}\leq 1} \E_F(f(X))-\E_{F_0}(f(X)).\label{eqn:MMD}
\end{align}

In the uncensored setting, the Maximum mean discrepancy has been used in a variety of testing problems. Indeed, in the simplest case, we can assess if our data points are generated from a distribution $F_0$ by comparing it with the empirical distribution $\tilde F_n$. By using the equivalency between Equation \eqref{eqn:mmddistanceran1231} and \eqref{eqn:MMD}, we deduce that  $MMD(\tilde{F}_n,F_0)^2$ is a  V-statistic. This fact allows us to easily derive the relevant asymptotic results to construct a statistical test. 

In the setting of right-censored data we study $MMD(\widehat{F}_n,F_0)^2$ using the Kaplan-Meier estimator $\widehat{F}_n$. By using Equations \eqref{eqn:mmddistanceran1231} and \eqref{eqn:MMD}, our test-statistic can be written as
\begin{align}
MMD(\widehat{F}_n,F_0)^2=\int_0^\tau\int_0^\tau K(x,y)(d\widehat{F}_n(x)-dF_0(x))(d\widehat{F}_n(y)-dF_0(y)).\nonumber
\end{align}
Notice that $MMD(\widehat{F}_n,F_0)^2$  coincides with $\beta$ defined in Equation~\eqref{eqn:defiBeta}. Hence, under the null hypothesis $H_0:F=F_0$, and Condition \ref{assu:degenerate}, Theorem \ref{thm:decompositionDege} states
\begin{align}
nMMD(\widehat{F}_n,F_0)^2=n\beta(\widehat{F}_n,F_0)\overset{\mathcal{D}}{\to}\int_0^\tau\frac{K'(x,x)}{1-G(x)}dF_0(x)+\Psi,\nonumber
\end{align} 
where $\Psi$ is as in Theorem~\ref{thm:decompositionDege}. Notice that Theorem \ref{thm:decompositionDege} does not require the {degeneracy condition} of Corollary~\ref{Corollary:Vstatsdege}.

\begin{figure}
        \includegraphics[width=\textwidth]{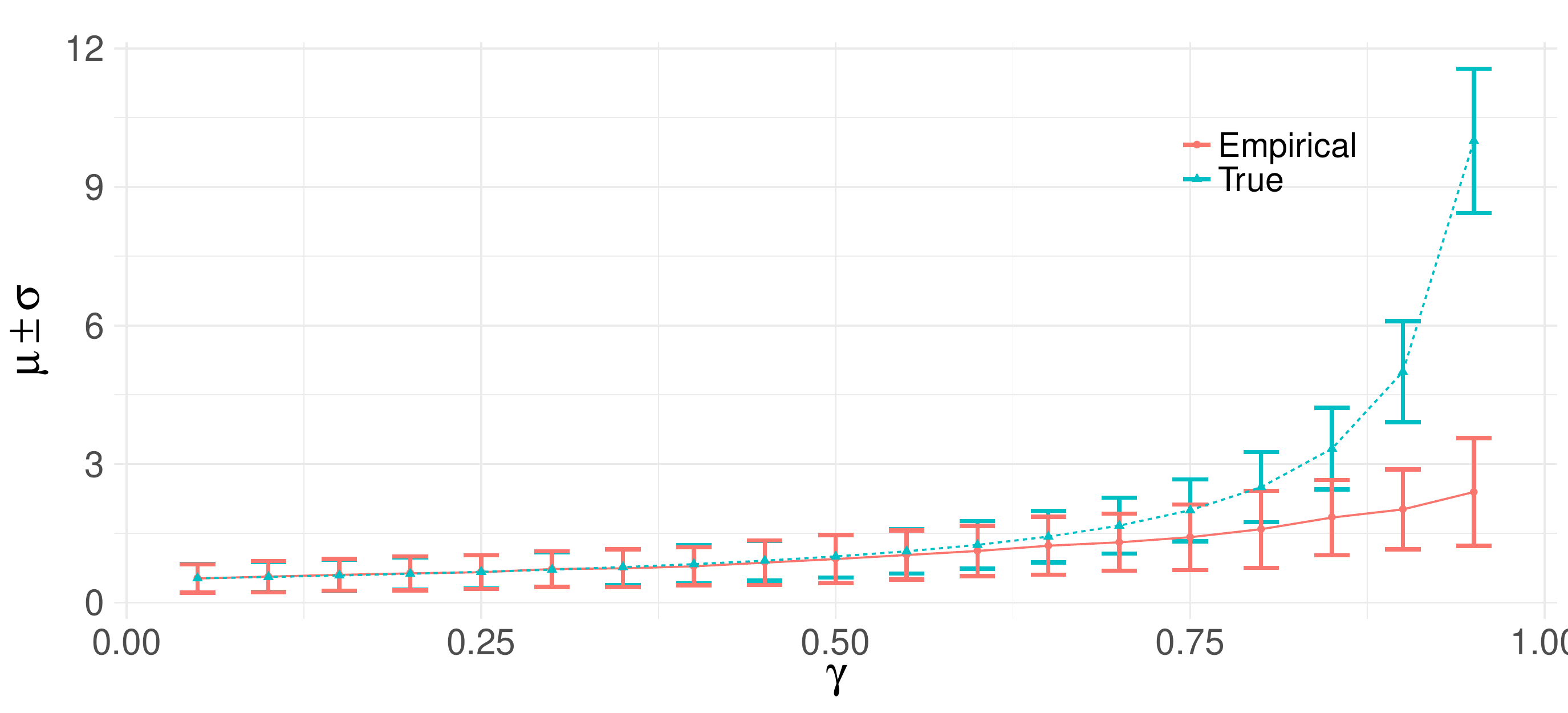}\caption{Empirical vs asymptotic mean of $nMMD(\widehat{F}_n,F_0)$ for different values of $\gamma<1$. For each value of the mean $\mu$ and $\widehat{\mu}$, we plot the intervals $[\mu-\sigma,\mu+\sigma]$ and $[\widehat\mu-\widehat\sigma,\widehat\mu+\widehat\sigma]$, where $\widehat\sigma^2$ and $\sigma^2$ denote the empirical and asymptotic variance respectively. We use a fixed sample size $n=3000$.}\label{Fig2}
\end{figure}

For the sake of simplicity, let us consider $K$ as the Ornstein-Uhlenbeck kernel given by $K(x,y) = e^{-|x-y|}$, and let $S_0(x) =e^{-x}$ and $1-G(x) = S_0(x)^{\gamma}$ (notice that for this choice of parameters $\tau=\infty$). A tedious computation shows that
\begin{align}
K'(x,y) = \begin{cases}
\frac{1}{2} (1-(x-y))e^{-(x-y)} & x>y\\
 \frac{1}{2} (1-(y-x))e^{-(y-x)} & x\leq y \\
 \end{cases}.\nonumber
\end{align}
Then, under the null hypothesis and Condition~\ref{assu:degenerate}, which  is satisfied for $\gamma<1$, it holds
\begin{align}
nMMD(\widehat{F}_n,F_0)^2\overset{\mathcal{D}}{\to}\frac{1}{2(1-\gamma)}+\Psi.\nonumber
\end{align}
Since $\E(\Psi)=0$, the asymptotic mean is given by $\frac{1}{2(1-\gamma)}$ and the asymptotic variance corresponds to 
\begin{align}
2\sum_{k=1}^\infty \lambda_k^2=2\int_{\R_+}\int_{\R_+} K'(x,y)^2e^{-(1-\gamma)(x+y)}dxdy=\frac{5-4\gamma+\gamma^2}{2(\gamma-3)^3(\gamma-1)}.\nonumber
\end{align}
In Figure \ref{Fig2}, we compare the empirical mean and variance with the mean and variance of the limit distribution. We repeat this experiment 1000 times for different values of $\gamma$ and a fixed sample size of 3000 data points. We observe that as $\gamma$ approaches 1 the empirical estimation starts to get far away from the mean and variance predicted by our result, suggesting a slow convergence rate.
\end{example}

\section{Conclusions and Final Remarks}
In this work we studied the limit distribution of Kaplan-Meier V- and U-statistics under two different regimes: the degenerate and the non-degenerate. Our results hold under very simple conditions and, in practice, we just need to check the finiteness of two simple integrals. Compared to previous approaches our results are much simpler to state and the conditions required to apply them are much easy to satisfy and verify. Additionally, our result gives more information about the limit distribution, e.g. we give closed-form  expressions for the asymptotic mean and variance, as well as an asymptotic canonical V-statistic representation of the Kaplan-Meier V-statistic.

We give a few comments about our results. First, in the canonical case (uncensored data), U-statistics are preferred over V-statistics due to several reasons. Arguably, the most important reason is that U-statistics are unbiased while V-statistics are, in general, biased. The bias of V-statistics implies that limit theorems need to deal with the behaviour of the biased part of the estimator, resulting in stronger conditions in the statement of the results. In the right-censored setting,  it does not seem to be a major difference between U- and V-statistics, and indeed, V-statistics are easier to work with as they can be represented by an integral with respect to the Kaplan-Meier estimator. Furthermore, due to the complex structure of the Kaplan-Meier weights, the Kaplan-Meier U-statistics are usually biased as opposed to its canonical counterpart, losing their main advantage over V-statistics.

Second, we think that our proof can be implemented in the settings of random kernels $K_n$ that depend on the data points $(X_i,\Delta_i)_{i=1}^n$ as long certain regularity conditions hold, namely, i) $K_n'$ is predictable in the sense of Definition~\ref{def:predicDsigmaAlgebra},  ii) it exists a deterministic kernel $\bar K$ such that $\sup_{x,y\leq \tau_n}|K_n(x,y)/{\bar K(x,y)}|= O_p(1)$, iii) $K_n$ converges in probability to some deterministic kernel $K$, iv) $\bar K$ and $K$ satisfy Conditions~\ref{assu:nondegenerate} or~\ref{assu:degenerate}, depending on the case of interest.

Third, our analysis can be extended to kernels of dimension greater than two by using the same underlying ideas exposed in this work. Nevertheless, the statements and proofs of the results become much more complicated due to long computations that come from the fact that the core of our proof strategy relies on decomposing the integration region into Interior and Exterior regions, thus, as the number of integrals grows so do the possible combinations of these type of regions. We do not include these type of results as they do not add much value to the current work, especially because U- and V-statistics of dimension two are the most common in  applications.

Finally, after the publication of the first preprint of this paper a few works have followed the path of using MMD distances to hypothesis testing in the setting of right-censored data. Particularly,  \citet{matabuena2019energy} implemented  our  MMD example for the two-sample problem, and extended it to Energy Distances, which are a generalisation of the MMD. Their analysis is a direct application of Theorems~\ref{thm:convergenceNondege} and~\ref{thm:decompositionDege}, and Corollary~\ref{Corollary:Vstatsdege}. In a similar direction, in \cite{fernandez2019reproducing} the authors studied  MMD distances in the context of hazard functions, obtaining as test-statistic a  double integral with respect to the Nelson-Aalen estimator. Due to the relationship between the Nelson-Aalen estimator and the counting process martingale $M(t)$, the asymptotic analysis of their test-statistic is carried out by using the techniques of this paper. Somewhat related is the work of \citet{rindt2019nonparametric}, where an MMD independent test for right-censored data is presented. Their test-statistic is a Kaplan-Meier quadruple integral, however, under their null hypothesis, such a test-statistic becomes the product two Kaplan-Meier double integrals, thus its asymptotic analysis follows from our results.

\section{Proofs I: Road Map}\label{sec:Road Map}
In order to keep our proofs as tidy as possible and to emphasise the key steps without the distraction of messy computations, we give a list of intermediate steps that are needed to carry out the proof of our main results. 

Recall that Equation~\eqref{eqn:expansionI-T} states $\theta(\widehat{F}_n)-\theta(F;\tau)= 2\alpha+\beta$, where $\alpha=\int_{0}^\tau\phi(x)d(\widehat{F}_n-F)(x)$ and $\phi(x)=\int_0^\tau K(x,y)dF(y)$. We analyse $\alpha$ and $\beta$ individually. 

\subsection{Treatment of {$\alpha$}{}}
We distinguish between two cases, when $\phi(x)$, defines in Equation~\eqref{eqn:defiphi}, satisfies the degeneracy condition stated in the Corollary \ref{Corollary:Vstatsdege} and when it does not. For the first case, observe that $\alpha=0$ holds trivially. For the second case, an application of the Central Limit Theorem (CLT) of  \citet{akritas2000} gives us the asymptotic behaviour of $\alpha$.

\begin{theorem}[\citet{akritas2000}]\label{Thm: Akritas}
Let $\phi:(0,\infty)\to \R$ be such that $\int_{0}^{\tau} \frac{\phi(x)^2}{1-G(x-)}dF(x) < \infty$, then 
\begin{align}
\sqrt{n} \int_{0}^{\tau} \phi(x)d(\widehat{F}_n- F)(x) \overset{\mathcal D}{\to} N(0,\sigma^2),\nonumber
\end{align}
where $\sigma^2 = \int_{0}^{\tau}\frac{S(s)}{1-H(s-)}(A\phi)^2(s)dF(s)<\infty$.\end{theorem}
Then, by applying the previous CLT, we obtain the following Corollary.

\begin{corollary}\label{cor:I1nondege}
Consider $\phi(x)=\int_{0}^{\tau}K(x,y)dF(y)$. Then, under  Condition~\ref{assu:nondegenerate}, $\int_0^\tau\frac{\phi(x)^2}{1-G(x-)}dF(x)<\infty$, and thus  
\begin{align*}
\sqrt{n}2\alpha \overset{\mathcal D}{\to} N(0,4\sigma^2),
\end{align*}
where $\sigma^2 = \int_{0}^{\tau}\frac{S(s)}{1-H(s-)}(A\phi)^2(s)dF(s)<\infty$.
\end{corollary}

\subsection{Treatment of {$\beta$}{}}
Let $R_1,R_2$ be two subsets of $\R_+$. Denote by $\beta_{R_1\times R_2}$, the integral
\begin{align*}
\beta_{R_1\times R_2}=\int_{R_1}\int_{R_2}K(x,y)d(\widehat{F}_n-F)(x)d(\widehat{F}_n-F)(y).
\end{align*}
 Observe that $\beta=\beta_{I_{H}\times I_{H}}=\beta_{I_{H}^2}$, and that $\beta$ can be decomposed into $\beta=\beta_{I^2}+\beta_{(I^2)^c}$, where $(I^2)^c = (E\times E) \cup (I\times E) \cup (E\times I)$. We recall that $I$, $E$ and  $I_H$ are defined in Section~\ref{sec:notationinteriorExterior}. To avoid extra parentheses, we write $I^{2c}$ instead of $(I^2)^c$.

In Section \ref{sec:Exterior}, we study the asymptotic properties of $\beta_{I^{2c}}$, obtaining as main result the following Lemma.

\begin{lemma}\label{lemma:exteriorN}
Under Condition~\ref{assu:nondegenerate}, it holds that
$\sqrt{n}\beta_{I^{2c}} = o_p(1)$,
and under Condition~\ref{assu:degenerate}, it holds that
$n\beta_{I^{2c}} = o_p(1).$
\end{lemma}

The handling of the term $\beta_{I^2}$ is far more complicated since it contains all the important information about the limit distribution. In Section~\ref{sec:InteriorSec}, we transform $\beta_{I^2}$ into a more tractable object by performing a \textit{change of measure}, where instead of integrating with respect to $d(\widehat{F}_n-F)$, we integrate with respect to the measure $dM = dN -Yd\Lambda$. This is done by using Duhamel's equations (Proposition \ref{Prop:Duhamel}).  The main result of Section~\ref{sec:InteriorSec} is the following.

\begin{lemma}\label{lemma:nKMintoMartingale}
Under Condition~\ref{assu:nondegenerate}, it holds that 
\begin{align}
\sqrt{n}\beta_{I^2} = \sqrt{n}\int_{I^2} \frac{\widehat S_n(x-) \widehat S_n(y-)}{Y(x)Y(y)}K'(x,y)dM(x)dM(y)+o_p(1),\label{eqn:MartingaleRepresqrtn}
\end{align}
and under Condition~\ref{assu:degenerate}, it holds that
\begin{align}
n\beta_{I^2} = n\int_{I^2} \frac{\widehat S_n(x-) \widehat S_n(y-)}{Y(x)Y(y)}K'(x,y)dM(x)dM(y)+o_p(1),\label{eqn:MartingaleRepren}
\end{align}
where the kernel $K'$ is defined in Equation~\eqref{eqn:defK'}.
\end{lemma}

\subsection{Proof of Theorem \ref{thm:convergenceNondege}}

In order to prove  Theorem \ref{thm:convergenceNondege}, we require the following intermediate result, which is formally proved in Section~\ref{sec:Interior}.

\begin{lemma}\label{lemma:sqrtDMarting=0}
Under Condition~\ref{assu:nondegenerate} it holds that
\begin{align*}
\sqrt{n}\int_{I^2} \frac{\widehat S_n(x-) \widehat S_n(y-)}{Y(x)Y(y)}K'(x,y)dM(x)dM(y) = o_p(1).
\end{align*}
\end{lemma}

\begin{proof}[\textbf{Proof of Theorem \ref{thm:convergenceNondege}}]
Under Condition \ref{assu:nondegenerate}, Lemmas~\ref{lemma:exteriorN},~\ref{lemma:nKMintoMartingale} and~\ref{lemma:sqrtDMarting=0} prove $\sqrt{n}\beta \overset{\Prob}{\to} 0$, which together with  Corollary~\ref{cor:I1nondege} concludes the result.
\end{proof}

\subsection{Proof of Theorem \ref{thm:decompositionDege}}

We proceed to give proof to Theorem \ref{thm:decompositionDege} and Corollary \ref{Corollary:Vstatsdege}. Observe that under Condition \ref{assu:degenerate}, Lemma \ref{lemma:exteriorN}, and Equation~\eqref{eqn:MartingaleRepren} of Lemma~\ref{lemma:nKMintoMartingale} yield
\begin{align}
n\beta = \frac{1}{n}\int_{I^2} \frac{\widehat{S}_n(x-)\widehat{S}_n(y-)}{(Y(x)/n)( Y(y)/n)}K'(x,y)dM(x)dM(y)+o_p(1).\label{eqn:nbetaDecomposition12351}
\end{align}

Theorem \ref{lemma:afterMartingalenb}, states that  $\widehat{S}_n(x-)$ and $Y(x)/n$ in Equation~\eqref{eqn:nbetaDecomposition12351} can be substituted by their respective limits,  $S(x-)$ and $1-H(x-)$, obtaining that 
\begin{align}
n\beta=\frac{1}{n}\sum_{i=1}^n\sum_{j=1}^n
\int_{I_{H}^2} \frac{K'(x,y)}{(1-G(x-))(1-G(y-))}dM_i(x)dM_j(y)+o_p(1).
\label{eqn:randomtjtj48191}
\end{align}
The proof of Theorem  \ref{thm:decompositionDege} follows by noticing that the leading term in Equation~\eqref{eqn:randomtjtj48191} is a degenerate V-statistic.

\begin{proof}[\textbf{Proof of Theorem \ref{thm:decompositionDege}}]
Equation~\eqref{eqn:randomtjtj48191} states
\begin{align}
\beta = \frac{1}{n^2} \sum_{i=1}^n \sum_{j=1}^n J((X_i,\Delta_i),(X_j,\Delta_j)) + o_p(n^{-1}),\nonumber
\end{align}
where $J$ is the kernel defined in Equation~\eqref{eqn:Jfunction}, thus we deduce that $\beta$ is a canonical V-statistic up an error of order $o_p(n^{-1})$. From Condition \ref{assu:degenerate}, we can deduce that  $\E(J((X_1,\Delta_1),(X_2,\Delta_2))^2)<\infty$, and  that $J$ satisfies the following degeneracy condition
\begin{align*}
\E(J((X_i,\Delta_i),(x,r)))& = 0,\nonumber
\end{align*}
for all $(x,r) \in I_H\times\{0,1\}$ (see  \Appendice~\ref{sec:UStatsRepBeta}). Therefore, by applying standard results for degenerate V-statistics, e.g. \citep[Theorem 4.3.2.]{Koroljuk94}, it holds
\begin{align}
n\beta  \overset{\mathcal D}{\to} \E(J((X_1,\Delta_1),(X_1,\Delta_1)))+\Psi,\nonumber
\end{align}
where $\Psi = \sum_{i=1}^{\infty}\lambda_i (\xi_i^ 2-1)$, $\xi_1,\xi_2,\ldots$ are i.i.d. standard normal random variables, and the $\lambda_i$'s are the eigenvalues of the integral operator $T_J:L^2(X,\Delta) \to L^2(X,\Delta)$ associated with $J$.  

Finally, note that $\E(\Psi) = 0$, and that 
\begin{align*}
\E(\Psi^2) = 2\sum_{i=1}^{\infty} \lambda_i^2&=2\E( J((X_1,\Delta_1),(X_2,\Delta_2))^2)\\
&=2\int _0^{\tau} \int_{0}^{\tau} \frac{S(x)S(y)K'(x,y)^2}{(1-H(x-))(1-H(y-))}dF(x)dF(y),
\end{align*}
where the last equality is formally verified in \Appendice~\ref{sec:UStatsRepBeta}.
\end{proof}

The rest of the paper is devoted to prove Lemmas \ref{lemma:exteriorN}, \ref{lemma:nKMintoMartingale} and \ref{lemma:sqrtDMarting=0}, and Theorem \ref{lemma:afterMartingalenb}.

\section{Proofs II: Preliminary Results}
The following results are going to be used several times in this paper.
\subsection{Some Results for Counting Processes}

\begin{proposition}\label{lemma:supConverSH}
The following results holds a.s.
\begin{enumerate}[i)]
\item $\lim_{n\to \infty} \sup_{t\leq \tau} |\widehat{S}_n(t)-S(t)| = 0$,
\item $\lim_{n \to \infty}\sup_{t \leq \tau} |Y(t)/n-H(t-)| = 0$, and
\item for every fixed $t^\star>0$ such that $1-H(t^\star-)>0$,
\begin{align*}
\sup_{t\leq t^\star} \left |n\widehat{S}_n(t-)/Y(t)  -1/(1-G(t-))\right| \to 0.
\end{align*}
\end{enumerate}
\end{proposition}

Item i. is due to \citet{stute93TheStrong}, item ii. is the Glivenko-Cantelli theorem, and item iii. follows from the two previous items.

\begin{proposition}\label{prop:ProbBounds}
Let $\beta \in (0,1)$, then the following results hold true:
\begin{enumerate}[i)]
\item $\Prob(\widehat{S}_n(t) \leq \beta^{-1}S(t),\quad\forall t \leq \tau_n) \geq 1-\beta$,\label{prop:ProbBound1}
\item $\Prob\left(Y(t)/n \geq \beta\{1-H(t-)\},\quad\forall t\leq\tau_n\right)\geq1-e(1/\beta) e^{-1/\beta}$,  \label{prop:ProbBound2}
\item $\Prob\left(Y(t)/n \leq \beta^{-1}\{1-H(t-)\},\quad\forall t\leq\tau_n\right)\geq1-\beta$, and\label{prop:ProbBound3}
\item $\Prob \left( n(1-H(\tau_n))\leq \beta^{-1}\right) \geq 1-e^{-\beta^{-1}}.$ \label{prop:ProbBound4}
\end{enumerate}
i.e. i) $\sup_{t\leq \tau_n} \frac{\widehat S_n(t)}{S(t)} = O_p(1)$, ii) and  iii) $\sup_{t\leq \tau_n}Y(t)/(n(1-H(t-)) = \Theta_p(1)$, and iv) $n(1-H(\tau_n)) = O_p(1)$.
\end{proposition}

Items i. and ii. are due to  \citet{gill1983large}. Item iii. is from \cite[Theorem 3.2.1]{gill1980censoring}, and Item iv. is due to \citet{Yang1994}.

Yet another useful result is the so-called Duhamel's Equation.
\begin{proposition}[Prop. 3.2.1 of \citep{gill1980censoring}]\label{Prop:Duhamel}
For all $x>0$ such that $S(x)>0$, 
\begin{align}\label{eqn:Diffduhamel}
d\widehat{F}_n(x) = dF(x)-\left(\int_{(0,x)} \frac{\widehat{S}_n(s-)}{S(s)Y(s)} dM(s)\right)dF(x)+ \frac{\widehat {S}_n(x-)}{Y(x)} dM(x).
\end{align}
\end{proposition}

Equation \eqref{eqn:Diffduhamel} allows us to deduce two important results.  Firstly,
\begin{align}\label{eqn:quickDhha}
&\int_{0}^{\tau_n} \phi(x) (d\widehat{F}_n(x)-dF(x))\nonumber\\
&\quad=\int_{0}^{\tau_n} \phi(x)\frac{\widehat {S}_n(x-)}{Y(x)} dM(x)- \int_{0}^{\tau_n}\int_{(0,x)}\phi(x)\frac{\widehat{S}_n(s-)}{S(s)Y(s)}dM(s)dF(x)\nonumber\\
&\quad=\int_{0}^{\tau_n} \frac{\widehat{S}_n(x-)}{Y(x)}\left(\phi(x)-\int_x^{\tau_n}\phi(s)\frac{dF(s)}{S(x)}\right) dM(x)\nonumber\\
&\quad=\int_{0}^{\tau_n} \frac{\widehat{S}_n(x-)}{Y(x)}(\widehat A\phi)(x)dM(x),
\end{align}
where $\widehat A$ is the operator defined in Section~\ref{sec:NotaAsu}. Secondly, for a 2-dimensional kernel $K$, we obtain
\begin{align}
&\int_{0}^{\tau_n}\int_{0}^{\tau_n} K(x,y)(d\widehat{F}_n(x)-dF(x))(d\widehat{F}_n(y)-dF(y))\nonumber\\
&\quad= \int_0^{\tau_n}\int_0^{\tau_n} \frac{\widehat S_n(x-)\widehat S_n(y-)}{Y(x)Y(y)} (\widehat A_2 \widehat A_1 K)(x,y) dM(x)dM(y).\label{eqn:quickDuha2}
\end{align}

\subsection{Some Convergence Theorems}

We state, without proof, the following elementary result that is useful to prove that a sequence of (random) integrals converge to zero in probability.

\begin{lemma}\label{lemma:integralConvergence2}
Let $(\mathcal X, \mathcal B, \mu)$ be a $\sigma$-finite measure space. Let $(R_n(x):x \in \mathcal X)$ be a sequence of stochastic process indexed on $\mathcal X$. We assume that $R_n(\cdot)$ is measurable with respect to $\mathcal B$ (for any fixed realisation of $R_n$).  Suppose that
\begin{enumerate}
\item[i)] For each $x \in \mathcal X$, $R_n(x) \to 0$ almost surely as $n$ tends to infinity, and
\item[ii)] it exists a deterministic non-negative function $R: \mathcal X \to [0,\infty)$ such that
\begin{align*}
\sup_{x \in \mathcal X} \frac{|R_n(x)|}{R(x)}=O_p(1),
\end{align*}
and that $\int_{x \in \mathcal X} R(x)\mu(dx)<\infty$. 
\end{enumerate}
Define the  sequence of random integrals $I_n = \int_{\mathcal X}R_n(x)\mu(dx)$, then  $I_n = o_p(1)$.
\end{lemma}

\subsection{Some Martingale Results}\label{sec:specialMartingales}

 For a given martingale $W$, we denote by $\langle W \rangle$ and $[W]$, respectively, the predictable and quadratic variation processes associated with $W$. It is particularly useful to remember that for counting process martingales $M_i$ and $M$ we have that 
\begin{align*}
d\langle M_i \rangle_t = (1-\Delta \Lambda(y))Y_i(t)d\Lambda(t) \text{ and  }\  d\langle M \rangle_t = (1-\Delta \Lambda(y))Y(t)d\Lambda(t),
\end{align*} 
and note that $1-\Delta \Lambda(y) = S(y)/S(y-)$.

  In our proofs we will constantly use the Lenglart-Rebolledo inequality  \citep[Theorem 3.4.1]{Flemming91}, in particular, we will use the fact that if $W$ is a submartingale depending on the number of observations $n$, with compensator $R$, then the
Lenglart-Rebolledo inequality implies that $\sup_{t\leq T} W(t) \overset{\Prob}{\to} 0$ for any stopping time $T$ such that $R(T) \overset{\Prob}{\to} 0$. Here limits are taken as $n$ approaches infinity. Throughout the proofs, we may not explicitly write the dependence on $n$ when writing a stochastic process, e.g. the martingale $M(t)$ depends on all data points.
 
In this work we will often encounter (sub)martingales with extra parameters, and we will integrate with respect to them. A particular case is stated in the following lemma, whose proof is very simple (and thus omitted).

\begin{lemma}\label{lemma:IntegralOfMartingales}
Let $(\mathcal X, \mathcal B, \mu)$ be a $\sigma$-finite measure space. Consider the stochastic process $(M_y(t): t \geq 0, y \in \mathcal X)$, and assume that 
\begin{enumerate}
\item[i)] For every fixed $y \in \mathcal X$, $M_y(t)$ is a square-integrable $(\mathcal F_t)$-martingale, and
\item[ii)] for every $t\geq 0$, $\E\left(\int_{\mathcal X} M_y(t)^2\mu(dy)\right)<\infty$.
\end{enumerate}
Then, for fixed $B \in \mathcal B$,  the stochastic process $W(t) = \int_B M_y(t)^2d\mu(y)$ is an $(\mathcal F_t)$-submartingale, and its compensator, $R$, is given by $R(t) = \int_B \langle M_y \rangle(t) d\mu(y)$.
\end{lemma}

Another interesting type of stochastic processes that appear in our proofs are double integrals with respect to martingales. Define the process 
\begin{align*}
W(t) = \int_{C_t} h(x,y)dM(x)dM(y),
\end{align*}
where $C_t = \{(x,y): 0< x <y\leq t\}$, and $C_0 = \emptyset$. The natural questions are whether $W(t)$ defines a proper martingale with respect to $(\mathcal F_t)_{t \geq 0}$ and, if that is the case, what is its predictable variation process (if it exists). We answer these questions below.

\begin{definition}
\label{def:predicDsigmaAlgebra}
Define the predictable $\sigma$- algebra $\mathcal{P}$ as the $\sigma$-algebra generated by the sets of the form
\begin{align*}
(a_1,b_1]\times(a_2,b_2]\times X, \text{ where }0\leq  a_1\leq b_1 < a_2 \leq b_2, X\in \mathcal{F}_{a_2},
\end{align*}
and $\{0\} \times \{0\}\times X$ with $X\in\mathcal{F}_0$.
\end{definition}
Let $C = \{(x,y): 0< x <y<\infty\}$. A process $(h(x,y): (x,y) \in C)$ is called elementary predictable if it can be written as a finite sum of indicator functions of sets belonging to the predictable $\sigma$-algebra $\mathcal{P}$. On the other hand, if a process $h$ is $\mathcal{P}$-measurable then it is the almost sure limit of elementary predictable functions.

Straightforwardly from Definition~\ref{def:predicDsigmaAlgebra} we get the following proposition.
 
\begin{proposition}\label{coll:h predictable}
If $h_1(t)$ and $h_2(t)$ are predictable w.r.t. $(\mathcal F_t)_{t\geq 0}$, then $h(x,y) = h_1(x)h_2(y)$ is $\mathcal P$-measurable. Also, all deterministic functions are $\mathcal P$-measurable.
\end{proposition}

\begin{theorem}\label{thm:doubleMartingale}
Let $h$ be a $\mathcal P$-measurable process, and suppose that for all $t \geq 0$ it holds that
\begin{align}\label{eqn:DoubleMcond2}
\E\left(\int_{C_t} |h||dM(x)dM(y)|\right) < \infty.
\end{align}
Then, 
\begin{align*}
W(t) = \int_{C_t} h(x,y)dM(x)dM(y)
\end{align*}
is a martingale on $\R_+$ with respect to the filtration $(\mathcal F_t)_{t\geq 0}$.

Moreover, if 
\begin{equation}\label{eqn:doubleMComcond2}
\E\left(\int_{(0,t]} \left(\int_{(0,y)}|h(x,y)||dM(x)|\right)^2 d \langle M \rangle(y)\right) <\infty,
\end{equation}
then $W(t)$ is a square-integrable $(\mathcal F_t)$-martingale with predictable variation process $\langle W \rangle$ given by
\begin{align}
\langle W \rangle (t) = \int_{(0,t]} \left(\int_{(0,y)} h(x,y)dM(x)\right)^2\frac{Y(y)S(y)}{S(y-)^2}dF(y).\nonumber
\end{align}
\end{theorem}
The proof of Theorem~\ref{thm:doubleMartingale} is given in \Appendice~\ref{Appendix:DSI}.
\subsection{Forward Operators}\label{sec:ForBackOperators}

\begin{lemma}\label{Lemma:FiniteOpe}
Under Condition~\ref{assu:nondegenerate}, it holds that

\begin{align}
\int_0^{\tau}\int_0^{\tau}  \frac{(A_1K)^2(x,y)S(x)}{1-H(x-)}dF(x)dF(y)<\infty,\label{eqn:FiniteOpe1}
\end{align}
and that
\begin{align}
\int_0^{\tau}\int_0^{\tau}  \frac{(A_1A_2K)^2(x,y)S(x)S(y)}{1-H(x-)S(y-)}dF(x)dF(y)<\infty,\label{eqn:FiniteOpe4}
\end{align}

Under Condition~\ref{assu:degenerate}, it holds that

\begin{align}
\int_0^{\tau}\int_0^{\tau}  \frac{(A_1K)^2(x,y)S(x)}{(1-H(x-))(1-G(y-))}dF(x)dF(y)<\infty,\label{eqn:FiniteOpe2}
\end{align}
and that
\begin{align}\label{eqn:FiniteOpe3}
\int_0^{\tau}\int_0^{\tau}  \frac{(A_1A_2K)^2(x,y)S(x)S(y)}{(1-H(x-))(1-H(y-))}dF(x)dF(y)<\infty.
\end{align}
\end{lemma}

The proof of Lemma~\ref{Lemma:FiniteOpe} is given in \Appendice~\ref{apendix:FuturePast}.

\section{Proofs III: Exterior Region}\label{sec:Exterior}
\subsection{Proof of Lemma \ref{lemma:exteriorN}}
In this section we prove Lemma \ref{lemma:exteriorN}. Recall  that $I^{2c} = (E\times E) \cup (E\times I) \cup (I \times E)$, and, by the symmetry of the kernel $K$, $\beta_{I^{2c}}=\beta_{E\times E}+2\beta_{E\times I}$. Then, the result holds by Lemma \ref{lemma:psiR4=0} which states $\beta_{E\times E} = o_p(n^{-1/2})$ under Condition \ref{assu:nondegenerate}, and $\beta_{E \times E} = o_p(n^{-1})$ under Condition \ref{assu:degenerate}, and by Lemmas \ref{lemma:npsi1-2R2=01} and \ref{lemma:psi1-2R2=01}, which state  that $\beta_{E\times I}=o_p(n^{-1/2})$ under Condition \ref{assu:nondegenerate}, and $\beta_{E\times I}=o_p(n^{-1})$ under Condition \ref{assu:nondegenerate}, respectively.

In the rest of Section \ref{sec:Exterior}, we enunciate and prove Lemmas \ref{lemma:psiR4=0}, \ref{lemma:npsi1-2R2=01} and \ref{lemma:psi1-2R2=01}.

\begin{lemma}\label{lemma:psiR4=0}
Under Condition~\ref{assu:nondegenerate}, it holds that
$\beta_{E\times E} = o_p(n^{-1/2})$. Moreover, under Condition~\ref{assu:degenerate}, it holds that $\beta_{E \times E} = o_p(n^{-1})$.

\end{lemma}
\begin{proof}
First, we prove $\beta_{E \times E}=o_p(n^{-1/2})$ under Condition~\ref{assu:nondegenerate}. Observe that $d\widehat F_n(t)=0$ for all $t>\tau_n$, thus $d(\widehat{F}-F)(x)d(\widehat{F}-F)(y)=dF(x)dF(y)$ in the region ${E \times E}$. Then, by the Cauchy-Schwartz's inequality, it holds
\begin{align}
\beta_{E \times E}=\int_{\tau_n}^\tau\int_{\tau_n}^\tau K(x,y)dF(x)dF(y)
\leq S(\tau_n)\sqrt{\int_{\tau_n}^\tau\int_{\tau_n}^\tau K(x,y)^2dF(x)dF(y)}.\nonumber
\end{align}

Multiplying by $\sqrt{(1-G(\tau_n))/(1-G(\tau_n))}=1$, we get
\begin{align*}
\sqrt{n}\beta_{E \times E}&\leq \sqrt{n\frac{(1-H(\tau_n))S(\tau_n)}{1-G(\tau_n)}\int_{\tau_n}^\tau\int_{\tau_n}^\tau K(x,y)^2dF(x)dF(y)}\nonumber\\
&\leq\sqrt{n(1-H(\tau_n))S(\tau_n)\int_{\tau_n}^\tau\int_{\tau_n}^\tau \frac{K(x,y)^2}{1-G(x-)}dF(x)dF(y)}=o_p(1),
\end{align*}
where the last equality follows from the facts that $n(1-H(\tau_n))=\bigOp(1)$ by Proposition \ref{prop:ProbBounds}.\ref{prop:ProbBound4}, and that the double integral tends to $0$ since $\tau_n \to \tau$ when $n$ tends to infinity, and by Condition \ref{assu:nondegenerate}.

Following the same argument, under Condition~\ref{assu:degenerate}, we get
\begin{align*}
n\beta_{E \times E}\leq n(1-H(\tau_n))\sqrt{\int_{\tau_n}^\tau\int_{\tau_n}^\tau \frac{K(x,y)^2dF(x)dF(y)}{(1-G(x-))(1-G(y-))}}=o_p(1),
\end{align*}
since $n(1-H(\tau_n))=\bigOp(1)$, and since the double integral tends to 0 by Condition~\ref{assu:degenerate}, together with the fact that $\tau_n\to\tau$.
\end{proof}

\begin{lemma}\label{lemma:npsi1-2R2=01}
Under Condition~\ref{assu:degenerate}, it holds that $\beta_{E\times I} = o_p(n^{-1}).$

\end{lemma}
\begin{proof}
We start by noticing that if $\tau$ is a point of discontinuity of $H$ then $\tau_n=\tau$ almost surely for a sufficiently large $n$. Consequently, the set $E\times I$ is empty and thus  the statement above holds trivially. Therefore, we assume that $\tau$ is a continuity point of $H$.

Replacing Equation~\eqref{eqn:quickDhha} in $\beta_{E\times I}$ yields
\begin{align}
|n\beta_{E\times I}|
&=\left| n\int_{\tau_n}^\tau \int_{0}^{\tau_n} \frac{\widehat S_n(x-)}{Y(x)} \left(\widehat A_1 K\right)(x,y) dM(x) dF(y)\right|\nonumber\\
&=\left|n\int_{\tau_n}^\tau \int_{0}^{\tau_n} \frac{\widehat S_n(x-)}{Y(x)} \left(\left(\widehat A_1-A_1 +A_1\right)K\right)(x,y) dM(x) dF(y)\right|.\nonumber
\end{align} 
Recall that Equation~\eqref{eqn:diffAes} states that for $x\leq \tau_n$, 
\begin{align}
\left(\left(\widehat A_1- A_1\right)K\right)(x,y)&= \frac{1}{S(x)}\int_{\tau_n}^{\tau} K(s,y)dF(s),\nonumber
\end{align}
then
\begin{align}
&n\int_{\tau_n}^\tau \int_0^{\tau_n}\frac{\widehat S_n(x-)}{Y(x)} \left(\left(\widehat A_1-A_1\right)K\right)(x,y) dM(x) dF(y)\nonumber\\
&\quad= n\int_{\tau_n}^\tau  \int_{0}^{\tau_n}\frac{\widehat S_n(x-)}{S(x)Y(x)} \int_{\tau_n}^{\tau}K(s,y)dF(s) dM(x) dF(y)\nonumber\\
&\quad=L(\tau_n)\left(n \int_{\tau_n}^\tau  \int_{\tau_n}^{\tau}K(s,y)dF(s)  dF(y)\right)
=L(\tau_n) \beta_{E\times E} = o_p(1)\label{eqn:randomveiei394811},
\end{align}
where we define $L(\tau_n)=\int_{0}^{\tau_n}\frac{\widehat S_n(x-)}{S(x)Y(x)}dM(x)$. To verify the last equality, we use  Lemma~\ref{lemma:psiR4=0},
and the fact that $L(\tau_n) = O_p(1)$, which follows from Lemma 2.4 of \citet{gill1983large} that states that $L(\tau_n) = 1- \widehat S_n(\tau_n)/S(\tau_n)$, and then by Proposition~\ref{prop:ProbBounds}.\ref{prop:ProbBound1}, we get $L(\tau_n) = O_p(1)$. 

Therefore, from Equation \eqref{eqn:randomveiei394811}, we deduce
\begin{align}
|n\beta_{E\times I}| &= \left|n\int_{\tau_n}^\tau\int_0^{\tau_n}
\frac{\widehat S_n(x-)}{Y(x)}(A_1K)(x,y)dM(x)dF(y)\right|+o_p(1)\label{eqn:Psi12-Psi22recap0}\\
&= \left|n\int_{\tau_n}^\tau M^\star_y(\tau_n)dF(y)\right|+o_p(1),\nonumber
\end{align}
where 
\begin{align*}
M^\star_y(t) = \int_0^t \frac{\widehat S_n(x-)}{Y(x)}(A_1K)(x,y)dM(x).
\end{align*}
We proceed to show that $n\int_{\tau_n}^\tau M^\star_y(\tau_n)dF(y)=o_p(1)$ which implies that  $n\beta_{E\times I} =o_p(1)$. 

Notice that for any fixed $y\in\R_+$, $M_y^\star(t)$ is a square-integrable $(\mathcal{F}_t)$-martingale. By applying the Cauchy-Schwartz's inequality, we obtain
\begin{align}
\left|n\int_{\tau_n}^\tau M^\star_y(\tau_n)dF(y)\right|
&\leq n^{1/2}S(\tau_n)^{1/2}\left(n\int_{\tau_n}^\tau M^\star_y(\tau_n)^2dF(y)\right)^{1/2}\nonumber\\
&\leq n^{1/2}\frac{(1-H(\tau_n))^{1/2}}{(1-G(\tau_n))^{1/2}}\left(n\int_{\tau_n}^\tau M^\star_y(\tau_n)^2dF(y)\right)^{1/2}\nonumber\\
&\leq n^{1/2}(1-H(\tau_n))^{1/2}\left(n\int_{\tau_n}^\tau \frac{M^\star_y(\tau_n)^2}{1-G(y-)}dF(y)\right)^{1/2}\nonumber\\
&= O_p(1)\left(n\int_{\tau_n}^\tau \frac{M^\star_y(\tau_n)^2}{1-G(y-)}dF(y)\right)^{1/2},\nonumber
\end{align}
where the  last equality follows from Proposition \ref{prop:ProbBounds}.\ref{prop:ProbBound4}. We proceed to prove that
\begin{align}
n\int_{\tau_n}^\tau \frac{M^\star_y(\tau_n)^2}{1-G(y-)}dF(y)=o_p(1).\nonumber
\end{align}
Notice that the previous equation considers random integration limits. Our first step will be to prove that $\tau_n$ can be replaced by a deterministic value, say $T_n$, without affecting the result we wish to prove. 

Let $C>0$ be a large constant, define $T_n = \inf\{t>0: H(t)\geq1-C/n\}$ and the event $B_n= \{1-C/n \leq H(\tau_n)\}$. By Proposition~\ref{prop:ProbBounds}.\ref{prop:ProbBound4}, it holds $\Prob(B_n) \geq 1-e^{-C}$ and, by the definition of $T_n$, we have that $\{\tau_n \geq T_n\} \subseteq B_n$. Since $\lim_{C \to \infty} \Prob(B_n)=1$, it is enough to prove that 
\begin{align*}
n\int_{T_n}^\tau \frac{M^\star_y(\tau_n)^2}{1-G(y-)}dF(y)=o_p(1).
\end{align*}

Observe that by Lemma \ref{lemma:IntegralOfMartingales}, the process
\begin{align}
t \to n\int_{T_n}^\tau\frac{M_y^\star(t)^2}{1-G(y-)}dF(y)\nonumber
\end{align}
is an $(\mathcal F_t)$-submartingale with compensator, evaluated at $t=\tau_n$, given by
\begin{align}
n\int_{T_n}^\tau\frac{\langle M_y^\star\rangle(\tau_n)}{1-G(y-)}dF(y)
&= n\int_{T_n}^\tau\int_0^{\tau_n}
(A_1K)^2(x,y) \frac{\widehat S_n(x-)^2S(x)}{Y(x)S(x-)^2}\frac{dF(x)dF(y)}{(1-G(y-))}\nonumber\\
&={O}_p(1)
\int_{T_n}^\tau\int_0^{\tau_n}
(A_1K)^2(x,y)\frac{S(x)}{(1-H(x-))}\frac{dF(x)dF(y)}{(1-G(y-))}\nonumber\\
&=o_p(1),\nonumber
\end{align}
where the second equality is due to Propositions~\ref{prop:ProbBounds}.\ref{prop:ProbBound1} and \ref{prop:ProbBounds}.\ref{prop:ProbBound2}, and the third equality holds by noticing that $T_n \to \tau$ and that
\begin{align}
&\int_{0}^\tau\int_0^{\tau}
(A_1K)^2(x,y)\frac{S(x)}{(1-H(x-))}\frac{dF(x)dF(y)}{(1-G(y-))}<\infty,\nonumber
\end{align}
by Equation~\eqref{eqn:FiniteOpe2} of Lemma~\ref{Lemma:FiniteOpe} under Condition \ref{assu:degenerate}. We conclude then that 
\begin{align*}
n\int_{T_n}^\tau\frac{\langle M_y^\star\rangle(\tau_n)}{1-G(y-)}dF(y)=o_p(1),
\end{align*}
 which, by the  Lenglart-Rebolledo inequality, implies 
\begin{align*}
n\int_{T_n}^\tau\frac{M_y^\star(\tau_n)^2}{1-G(y-)}dF(y)=o_p(1).
\end{align*} 
Since the previous result is valid in the event $B_n$, which can be chosen with arbitrarily large probability, we conclude
\begin{align*}
n\beta_{E\times I}&=O_p(1)\left(n\int_{\tau_n}^\tau \frac{M^\star_y(\tau_n)^2}{1-G(y-)}dF(y)\right)^{1/2}+o_p(1)=o_p(1),
\end{align*}
finishing our proof.
\end{proof}

\begin{lemma}\label{lemma:psi1-2R2=01}
Under Assumption~\ref{assu:nondegenerate} it holds that
$\beta_{E\times I} = o_p(n^{-1/2}).$
\end{lemma}

\begin{proof}
Following the same steps of the proof of Lemma~\ref{lemma:npsi1-2R2=01}, it holds
\begin{align}
n^{1/2}\beta_{E\times I} &=n^{1/2}\int_{\tau_n}^\tau\int_0^{\tau_n}
\frac{\widehat S_n(x)}{Y(x)}(A_1K)(x,y)dM(x)dF(y)+o_p(1)\nonumber\\
&= n^{1/2}\int_{\tau_n}^\tau M^\star_y(\tau_n)dF(y)+o_p(1).\nonumber
\end{align}
where $M^\star_y(t)=\int_0^t\frac{\widehat S_n(x-)}{Y(x)}(A_1K)(x,y)dM(x)$ is an $(\mathcal{F}_t)$-martingale for every fixed $y\in\R_+$. Let $T_n$ be the same deterministic sequence used in the proof of Lemma~\ref{lemma:npsi1-2R2=01}. Then, it suffices to show that
\begin{align}
n^{1/2}\int_{T_n}^\tau M_y^\star(\tau_n)dF(y)=o_p(1).\nonumber
\end{align}

By the Cauchy-Schwartz's inequality 
\begin{align}
|n^{1/2}\beta_{E\times I}|&\leq n^{1/2}\int_{T_n}^\tau M_y^\star(\tau_n)dF(y)+o_p(1)\nonumber\\
&\leq\left(n\int_{T_n}^\tau M_y^\star(\tau_n)^2dF(y)\right)^{1/2}+o_p(1).\label{eqn:ExICauchy}
\end{align}

Moreover, notice that by Lemma \ref{lemma:IntegralOfMartingales}, $n\int_{T_n}^\tau M_y^\star(t)^2dF(y)$ is an $(\mathcal{F}_t)$-submartingale with compensator, evaluated at $t=\tau_n$, given by
\begin{align}
n\int_{T_n}^\tau \langle M_y^\star\rangle(\tau_n)dF(y)&=\int_{T_n}^\tau \int_0^{\tau_n}\frac{\widehat{S}(x-)^2}{Y(x)/n}(A_1K)^2(x,y)\frac{S(x)}{S(x-)^2}dF(x)dF(y)\nonumber\\
&={O}_p(1)\int_{T_n}^\tau \int_0^{\tau_n}(A_1K)^2(x,y) \frac{S(x)}{1-H(x-)}dF(x)dF(y),\nonumber
\end{align}
where the second equality holds by Propositions~\ref{prop:ProbBounds}.\ref{prop:ProbBound1} and \ref{prop:ProbBounds}.\ref{prop:ProbBound2}. We prove that the compensator in the previous equation converges to zero by noticing that $T_n\to\tau$, and that
\begin{align*}
&\int_{0}^\tau \int_0^{\tau}(A_1K)^2(x,y) \frac{S(x)}{1-H(x-)}dF(x)dF(y)<\infty,
\end{align*}
which holds due to Equation~\eqref{eqn:FiniteOpe1} of Lemma~\ref{Lemma:FiniteOpe} under Condition \ref{assu:nondegenerate}. 

The previous result implies that $n\int_{T_n}^\tau \langle M_y^\star\rangle(\tau_n)dF(y)=o_p(1)$. By the Lenglart-Rebolledo inequality, we deduce $n\int_{T_n}^\tau M_y^\star(\tau_n)^2dF(y)=o_p(1),$ and by substituting this result in Equation~\eqref{eqn:ExICauchy} we get  $n^{1/2}\beta_{E\times I}=o_p(1)$.

\end{proof}

\section{Proofs IV: Interior Region}\label{sec:InteriorSec}
\subsection{Proof of Lemma \ref{lemma:nKMintoMartingale}}
\begin{proof}
We start by proving Equation \eqref{eqn:MartingaleRepren} under Condition~\ref{assu:degenerate}. Observe that, by Equation~\eqref{eqn:quickDuha2}, it holds that
\begin{align}
\beta_{I^2} =  \int_0^{\tau_n}\int_0^{\tau_n} \frac{\widehat S_n(x-)\widehat S_n(y-)}{Y(x)Y(y)} (\widehat A_2 \widehat A_1 K)(x,y) dM(x)dM(y).\nonumber
\end{align}
We proceed to prove that $\widehat{A}_1$ and $\widehat{A}_2$ can be replaced in the previous equation by the operators $A_1$ and $A_2$, respectively.  After that, Equation~\eqref{eqn:MartingaleRepren} follows immediately by recalling that $K'(x,y)=(A_1A_2K)(x,y)$.

We begin with the following equality,
\begin{align*}
\widehat{A}_2\widehat{A}_1-A_2A_1=(\widehat{A}_2-A_2)(\widehat{A}_1-A_1)+A_2(\widehat{A}_1-A_1)+A_1(\widehat{A}_2-A_2),
\end{align*}
then, by the symmetry of $K$,  we just need to prove that
\begin{align}
n\int_0^{\tau_n} \int_0^{\tau_n} \frac{\widehat S_n(x-) \widehat S_n(y-)}{Y(x)Y(y)}(A_1(\widehat{A}_2-A_2)K)(x,y)dM(x)dM(y)=o_p(1),\label{eqn:L5.41}
\end{align}
and
\begin{align}
n\int_0^{\tau_n} \int_0^{\tau_n} \frac{\widehat S_n(x-) \widehat S_n(y-)}{Y(x)Y(y)}((\widehat{A}_2-A_2)(\widehat{A}_1-A_1)K)(x,y)dM(x)dM(y)=o_p(1).\label{eqn:L5.42}
\end{align}
We begin by proving Equation \eqref{eqn:L5.41}. From Equation~\eqref{eqn:diffAes} we get
\begin{align}\label{eqn:randomv949j2bxcngs}
(A_1(\widehat{A}_2-A_2)K)(x,y)&=\frac{1}{S(y)}\int_{\tau_n}^\tau (A_1K)(x,s)dF(s).
\end{align}
Let $L(\tau_n)=\int_{0}^{\tau_n}\frac{\widehat{S}(x-)}{S(x)}\frac{d M(x)}{Y(x)}$, then substituting Equation~\eqref{eqn:randomv949j2bxcngs} in Equation \eqref{eqn:L5.41} yields
\begin{align*}
&n\int_0^{\tau_n} \int_0^{\tau_n} \frac{\widehat S_n(x-) \widehat S_n(y-)}{Y(x)Y(y)}(A_1(\widehat{A}_2-A_2)K)(x,y)dM(x)dM(y)\\
&\quad=n \int_0^{\tau_n} \int_0^{\tau_n} \frac{\widehat S_n(x-) \widehat S_n(y-)}{Y(x)Y(y)}\left(\frac{1}{S(y)}\int_{\tau_n}^{\tau}(A_1K)(x,s)dF(s)\right) dM(x)dM(y)\nonumber\\
&\quad= nL(\tau_n)\int_0^{\tau_n} \frac{ \widehat S_n(x-)}{Y(x)}\int_{\tau_n}^{\tau}(A_1K)(x,s)dF(s)dM(x)\nonumber\\
&\quad= O_p(1)n\left|\int_{\tau_n}^{\tau}\int_0^{\tau_n} \frac{ \widehat S_n(x-)}{Y(x)}(A_1K)(x,s)dM(x)dF(s)\right|=n|\beta_{E \times I}| + o_p(1),
\end{align*}\label{eqn:random19cifnalasd}
where the third equality holds as $L(\tau_n)=O_p(1)$ (which is proved in Lemma~\ref{lemma:npsi1-2R2=01}),  and the last equality is exactly Equation \eqref{eqn:Psi12-Psi22recap0}. Hence, by Lemma \ref{lemma:npsi1-2R2=01}, we deduce that Equation \eqref{eqn:L5.41} holds true.

For Equation \eqref{eqn:L5.42}, a similar computation yields 
\begin{align*}
&n\int_0^{\tau_n} \int_0^{\tau_n} \frac{\widehat S_n(x-) \widehat S_n(y-)}{Y(x)Y(y)}((\widehat{A}_2-A_2)(\widehat{A}_1-A_1)K)(x,y)dM(x)dM(y)\\
&\quad=nL(\tau_n)^2 
\int_{\tau_n}^{\infty}\int_{\tau_n}^{\infty}
K(s,t)dF(s)dF(t)\\
&\quad=O_p(1)n 
\left|\int_{\tau_n}^{\infty}\int_{\tau_n}^{\infty}
K(s,t)dF(s)dF(t)\right|= O_p(1)n|\beta_{E\times E}| = o_p(1)
\end{align*}
where the last equality holds by Lemma~\ref{lemma:psiR4=0}.

To finish the proof of Lemma~\ref{lemma:nKMintoMartingale}  we need to check Equation~\eqref{eqn:MartingaleRepresqrtn} holds under Condition~\ref{assu:nondegenerate}, which follows from repeating the same steps but replacing the scaling factor $n$ by $\sqrt{n}$. 
\end{proof}

\section{Proofs V: Double Stochastic Integral}\label{sec:Interior}

In this section we prove Lemma \ref{lemma:sqrtDMarting=0} and Theorem \ref{lemma:afterMartingalenb}. To begin with, from Lemmas \ref{lemma:exteriorN} and \ref{lemma:nKMintoMartingale}, we deduce that
\begin{eqnarray}\label{eqn:random1kjfjcz}
\beta = \int_{0}^{\tau_n}\int_{0}^{\tau_n} \frac{\widehat S_n(x-) \widehat S_n(y-)}{Y(x)Y(y)} K'(x,y) dM(x)dM(y)+o_p(n^{-c})
\end{eqnarray}
holds for $c = 1/2$ under Condition~\ref{assu:nondegenerate}, and for $c = 1$ under Condition~\ref{assu:degenerate}. The form of $\beta$ suggests that we  need to study the double stochastic integral process given by 
\begin{align}
Q(t) = \int_{0}^{t}\int_{0}^{t} \frac{\widehat S_n(x-) \widehat S_n(y-)}{Y(x)Y(y)} K'(x,y) dM(x)dM(y).\nonumber
\end{align}
 The strategy to study $Q(t)$ is to consider its decomposition into a diagonal and an off-diagonal term, and to analyse them individually. To this end, we define the sets $D(t) = \{(x,y): x=y, 0<x\leq t\}$ and $C(t) = \{(x,y): 0<x<y\leq t\}$, and define the processes
\begin{align}
Q_D(t) &= \int_{D(t)}\frac{\widehat S_n(x-) \widehat S_n(y-)}{Y(x)Y(y)} K'(x,y) dM(x)dM(y),\nonumber
\end{align}
and
\begin{align} 
Q_C(t) &= \int_{C(t)}\frac{\widehat S_n(x-) \widehat S_n(y-)}{Y(x)Y(y)} K'(x,y) dM(x)dM(y).\nonumber
\end{align}
Notice that $Q(t) = Q_D(t)+2Q_C(t)$ follows by the symmetry of $K'=(A_1A_2K)$.

The proofs of Lemma \ref{lemma:sqrtDMarting=0} and Theorem \ref{lemma:afterMartingalenb} are an immediate consequence of the following results concerning the process $Q(t)$.
\begin{lemma}\label{lemma:sqrtnQD}
Under Condition~\ref{assu:nondegenerate} it holds that $\sqrt{n}Q_D(\tau_n) \overset{\Prob}{\to} 0$.
\end{lemma}

\begin{lemma}\label{lemma:nQ_DintoMartingale}
Under Condition~\ref{assu:degenerate} it holds that
\begin{align}
nQ_D(\tau_n) &= \frac{1}{n}\int_{0}^{\tau_n}\frac{K'(x,x)}{(1-G(x-))^2}d[M](x)+o_p(1).\nonumber
\end{align}
\end{lemma}

\begin{lemma}\label{lemma:snNDQ}
Under Condition~\ref{assu:nondegenerate} it holds that $
\sqrt n Q_C(\tau_n) \overset{\Prob}{\to} 0.$
\end{lemma}

\begin{lemma}\label{lemma:transformDMC}
Under Condition~\ref{assu:degenerate} it holds that
\begin{align*}
nQ_C(\tau_n)&= \frac{1}{n}\int_{0}^{\tau_n}\int_{(0,y)} \frac{K'(x,y)}{(1-G(x-))(1-G(y-))}dM(x)dM(y) + o_p(1).
\end{align*}
\end{lemma}

\begin{proof}[\textbf{Proof of Lemma \ref{lemma:sqrtDMarting=0}}]
Starting from Equation~\eqref{eqn:random1kjfjcz}, we get $\sqrt{n}\beta=\sqrt{n}Q(\tau_n)+o_p(1)=\sqrt{n}Q_D(\tau_n)+2\sqrt{n}Q_C(\tau_n)+o_p(1)$. Then, the result follows from Lemmas 
\ref{lemma:sqrtnQD} and \ref{lemma:snNDQ}.
\end{proof}

\begin{proof}[\textbf{Proof of Theorem \ref{lemma:afterMartingalenb}}]

From Equation~\eqref{eqn:random1kjfjcz} it holds that $n\beta=nQ(\tau_n)+o_p(1)=nQ_D(\tau_n)+2nQ_C(\tau_n)+o_p(1)$. Then a direct application of Lemmas \ref{lemma:nQ_DintoMartingale}  and \ref{lemma:transformDMC} yields
\begin{align*}
n\beta&=nQ_D(\tau_n)+2nQ_C(\tau_n)+o_p(1)\\
&= \frac{1}{n}\int_{0}^{\tau_n}\int_{0}^{\tau_n} \frac{K'(x,y)}{(1-G(x-))(1-G(y-))}dM(x)dM(y) + o_p(1)\\
&=\frac{1}{n}\sum_{i=1}^n\sum_{j=1}^n\int_{0}^{\tau_n}\int_{0}^{\tau_n} \frac{K'(x,y)}{(1-G(x-))(1-G(y-))}dM_i(x)dM_j(y) + o_p(1).
\end{align*}
\end{proof}

It just remains to prove Lemmas \ref{lemma:sqrtnQD}, \ref{lemma:nQ_DintoMartingale},  \ref{lemma:snNDQ}, and  \ref{lemma:transformDMC}.

\subsection{Integral over Diagonal {$D(t)$}{}: Proof of Lemmas~\ref{lemma:sqrtnQD} and \ref{lemma:nQ_DintoMartingale}}\label{eqn:sectionIntegralonDiagonal}
 
 Observe that $Q_D(t)$ satisfies
\begin{align}
Q_D(t)= \int_{0}^{t} \frac{\widehat S_n(x-)^2}{Y(x)^ 2} K'(x,x) d[M](x).\nonumber
\end{align}
The latter can be checked by noticing that the measure $dM(x)dM(y)$ of a small square whose main diagonal goes from $(a,a)$ to $(b,b)$ is $(M(b)-M(a))^2$. When $b$ approaches $a$ from above, we have that $(M(b) - M(a))^2 \to (\Delta M(a))^2$ (the limit is well-defined for $M$). Since $M$ is the difference of two increasing processes we have that the number of discontinuities is at most countable, then $\int_0^{t} f(x)d[M](x)=\sum_{x\leq t} f(x)(\Delta M(x))^2$.

To analyse the process $Q_D(t)$, recall that $[M]$ is a submartingale with compensator given by $\langle M \rangle$. Thus, for any predictable process $H\geq 0$, $\int_0^t H(x)d[M](x)$ is a submartingale with compensator given by $\int_0^t H(x)d\langle M \rangle (x)$. Finally, by the Lenglart-Rebolledo inequality, if we have  $\int_0^{\tau_n} H(x)d\langle M \rangle (x) \overset{\Prob}{\to} 0$, then we get that $\int_0^{\tau_n} H(x)d[M](x) \overset{\Prob}{\to} 0$.

\begin{proof}[\textbf{Proof of Lemma~\ref{lemma:sqrtnQD}}]

Define the $(\mathcal{F}_t)$-submartingale 
\begin{align*}
W(t)=\int_{0}^{t} \frac{\widehat S_n(x-)^2}{Y(x)^ 2} |K'(x,x)| d[M](x),
\end{align*}
and observe that $|\sqrt{n}Q_D(t)|\leq\sqrt{n}W(t)$. Thus, it is enough to prove that $\sqrt{n}W(t)=o_p(1)$. Abusing notation, denote by $\langle W\rangle(t)$ the compensator of $W(t)$. Then we will prove that  $\sqrt{n}\langle W\rangle(\tau_n)=o_p(1)$, and thus,  by the Lenglart-Rebolledo inequality, we will get $\sqrt{n}W(\tau_n)=o_p(1)$.

Observe that
\begin{align}
\sqrt{n}\langle W\rangle(\tau_n)&=\sqrt{n}\int_{0}^{\tau_n} \frac{\widehat S_n(x-)^2}{Y(x)^ 2} |K'(x,x)| d\langle M \rangle(x)\nonumber\\
&= \sqrt{n}\int_{0}^{\tau_n} \frac{\widehat S_n(x-)^2}{Y(x)^2} |K'(x,x)| \frac{Y(x)S(x)}{S(x-)^2}dF(x)\nonumber\\
&=\int_{0}^{\tau_n} \frac{\widehat S_n(x-)^2}{S(x-)^2} \frac{S(x)}{Y(x)/\sqrt{n}}|K'(x,x)|dF(x)\nonumber\\
&=O_p(1)\int_{0}^{\tau} \frac{\ind_{\{x \leq\tau_n\}}}{\sqrt{Y(x)}}\sqrt{\frac{S(x)}{1-G(x-)}} |K'(x,x)| dF(x),\nonumber
\end{align}
where the fourth equality follows from Propositions \ref{prop:ProbBounds}.\ref{prop:ProbBound1} and \ref{prop:ProbBounds}.\ref{prop:ProbBound2}. Finally, we claim that $\int_{0}^{\tau} \frac{1}{\sqrt{Y(x)}}\sqrt{\frac{S(x)}{1-G(x-)}} |K'(x,x)| dF(x) = o(1)$. This is verified by applying Dominated Convergence Theorem. Indeed, notice that $ \frac{\ind_{\{x \leq\tau_n\}}}{\sqrt{Y(x)}} \to 0$ for each fixed $x \in I_H$, thus the integrand tends to zero. Moreover, by using that $Y(x)\geq 1$ for $x\leq \tau_n$, the integrand is bounded by an integrable function due to Condition~\ref{assu:nondegenerate}.
\end{proof}

\begin{proof}[\textbf{Proof of Lemma~\ref{lemma:nQ_DintoMartingale}}]
Observe that it is enough to prove that
\begin{align}
\frac{1}{n}\int_{0}^{\tau_n} \left|\frac{\widehat S_n(x-)^ 2}{Y(x)^2/n^2}-\frac{1}{(1-G(x-))^2}\right||K'(x,x)|d[M](x)=o_p(1).\nonumber
\end{align}
Write $U_n(x) =\left|\frac{\widehat S_n(x-)^ 2}{Y(x)^2/n^2}-\frac{1}{(1-G(x-))^2}\right|$, which is predictable w.r.t. $(\mathcal F_x)_{x\geq 0}$. Also, define the process $W(t)=\frac{1}{n}\int_{0}^{t} U_n(x)|K'(x,x)|d[M](x)$, and observe that it corresponds to an $(\mathcal{F}_t)$-submartingale. We prove that $W(\tau_n) = o_p(1)$ by using the Lenglart-Rebolledo inequality. For such, we have to prove that its compensator, which by abusing notation we denote by $\langle W\rangle$, satisfies $\langle W\rangle(\tau_n) = o_p(1)$. A simple computation shows
\begin{align*}
\langle W \rangle(\tau_n)&=\frac{1}{n}\int_{0}^{\tau_n}U_n(x)|K'(x,x)|d\langle M\rangle(x)\\
&=\frac{1}{n}\int_{0}^{\tau_n} U_n(x)|K'(x,x)|\frac{Y(x)S(x)}{S(x-)^2}dF(x)\\
&=O_p(1)\int_{0}^{\tau} \ind_{\{x\leq \tau_n\}}U_n(x)|K'(x,x)|\frac{(1-G(x-))S(x)}{S(x-)}dF(x)\\
&=o_p(1),
\end{align*}
where the third equality follows from Proposition \ref{prop:ProbBounds}.\ref{prop:ProbBound3}, and the last equality follows from applying Lemma~\ref{lemma:integralConvergence2}, whose conditions we proceed to verify: for the first condition, Proposition~\ref{lemma:supConverSH} yields $U_n(x)=\left|\frac{\widehat S_n(x-)^2}{Y(x)^2/n^2}-\frac{1}{(1-G(x-))^2}\right|\to 0$ for every fixed $x\in I_H$. For the second condition, set $R(x) = \frac{|K'(x,x)|S(x)}{1-H(x-)}$, which is integrable by Condition~\ref{assu:degenerate}, then, Propositions \ref{prop:ProbBounds}.\ref{prop:ProbBound1} and \ref{prop:ProbBounds}.\ref{prop:ProbBound2} give  $U_n(x)=O_p(1)\frac{1}{(1-G(x-))^2}$, uniformly on $x\leq \tau_n$. Then
\begin{align*}
&\ind_{\{x\leq \tau_n\}}U_n(x)|K'(x,x)|\frac{(1-G(x-))S(x)}{S(x-)}\\
&\quad=O_p(1)|K'(x,x)|\frac{S(x)}{(1-G(x-))S(x-)}=O_p(1)R(x),
\end{align*}
uniformly on $x\leq \tau$. 
\end{proof}

\subsection{Integral over Off-diagonal {$C(t)$}{}: Proof of Lemmas \ref{lemma:snNDQ} and  \ref{lemma:transformDMC}}

\begin{proof}[\textbf{Proof of Lemma~\ref{lemma:snNDQ}}]
From Theorem~\ref{thm:doubleMartingale}, $Q_C(t)$ is a square-integrable $(\mathcal F_t)$-martingale with mean 0. Then, by the Lenglart-Rebolledo inequality, it is enough to prove its predictable variation process, denoted by $\langle Q_C \rangle(t)$, satisfies $n\langle Q_C \rangle(\tau_n) = o_p(1)$.

From Theorem~\ref{thm:doubleMartingale} we have that  $n\langle Q \rangle(\tau_n)$ is equal to
\begin{align}
&n\int_{0}^{\tau_n}	\frac{\widehat S_n(y-)^2S(y)}{Y(y)S(y-)^2} \left(\int \displaylimits_{(0,y)}\frac{\widehat S_n(x-)}{Y(x)}K'(x,y)dM(x) \right)^2 dF(y)\nonumber\\
&\quad= O_p(1) \int_{0}^{\tau}	\frac{S(y)}{1-H(y-)} \left(\int_0^{\tau_n}\frac{\widehat S_n(x-)}{Y(x)}K'(x,y)\ind_{\{x<y\}}dM(x) \right)^2 dF(y)\nonumber\\
&\quad =O_p(1) \int_{0}^{\tau}	\frac{S(y)}{1-H(y-)} M_y^\star(\tau_n)^2 dF(y)\label{eqn:random1928dnfuf},
\end{align}
where the first equality is due to Propositions \ref{prop:ProbBounds}.\ref{prop:ProbBound1} and \ref{prop:ProbBounds}.\ref{prop:ProbBound2}, and in the second equality we define $M_y^\star(t)= \int_{0}^t\frac{\widehat S_n(x-)}{Y(x)}K'(x,y)\ind_{\{x<y\}}dM(x)$, which is a square-integrable $(\mathcal{F}_t)$-martingale for any fixed $y\in I_H$. 

Define the process $Z(t) = \int_{0}^{\tau}\frac{S(y)}{1-H(y-)}M_y^\star(t)^2dF(y)$, and notice that the integral in Equation~\eqref{eqn:random1928dnfuf} corresponds to $Z({\tau_n})$. By Lemma~\ref{lemma:IntegralOfMartingales},  $Z(t)$ is an $(\mathcal{F}_t)$-submartingale, hence we shall use the Lenglart-Rebolledo inequality to prove $Z({\tau_n}) = o_p(1)$ by proving that the compensator of $Z(t)$, which by abusing notation we denote by $\langle Z \rangle (t)$, satisfies $\langle Z \rangle(\tau_n) = o_p(1)$. From Lemma~\ref{lemma:IntegralOfMartingales}, and Propositions \ref{prop:ProbBounds}.\ref{prop:ProbBound1} and  \ref{prop:ProbBounds}.\ref{prop:ProbBound2}, we have that
\begin{align}
\langle Z \rangle(\tau_n)&=\int_{0}^{\tau}\int_{0}^{\tau_n}
\frac{S(y)}{1-H(y-)} \frac{\widehat S_n(x-)^2}{Y(x)}K'(x,y)^2\ind_{\{x<y\}}\frac{S(x)}{S(x-)^2}dF(x)dF(y)\nonumber\\
&=O_p(1) \int_{0}^{\tau}
\int_{0}^{\tau} \frac{\ind_{\{x \leq\tau_n\}}S(y)S(x)}{Y(x)(1-H(y-))}K'(x,y)^2\ind_{\{x<y\}}dF(x)dF(y).\nonumber
\end{align}
We claim that the previous quantity tends to 0 as $n$ approaches infinity by an application of the Dominated Convergence Theorem. Indeed, notice that $\frac{\ind_{\{x \leq\tau_n\}}}{Y(x)}\to 0$ for any fixed $x\in I_H$, and that 
\begin{align*}
\frac{\ind_{\{x \leq \tau_n\}}}{Y(x)}\frac{S(y)S(x)}{(1-H(y-))}K'(x,y)^2\ind_{\{x<y\}}&\leq  \frac{S(y)S(x)}{(1-H(y-))}K'(x,y)^2,
\end{align*}
which is integrable by Equation~\eqref{eqn:FiniteOpe4} of Lemma~\ref{Lemma:FiniteOpe} (Recall that $K' = A_1A_2K$).
\end{proof}

\begin{proof}[\textbf{Proof of Lemma~\ref{lemma:transformDMC}}]
The result follows from proving that
\begin{align}
\frac{1}{n}\int_{C(\tau_n)} \left(\frac{\widehat S_n(x-)\widehat S_n(y-)K'(x,y)}{(Y(x)/n)(Y(y)/n)}-\frac{K'(x,y)}{(1-G(x-))(1-G(y-))}\right)dM(x)dM(y)\nonumber
\end{align}
is $o_p(1)$, which is equivalent to prove that
\begin{align}
\frac{1}{n}\int_{0}^{\tau_n}\int_{(0,y)} \left(\frac{\widehat S_n(x-)}{Y(x)/n}-\frac{1}{1-G(x-)}\right)\frac{\widehat S_n(y-)K'(x,y)}{Y(y)/n}dM(x)dM(y)= o_p(1),\label{eqn:Off1}
\end{align}
and that
\begin{align}
\frac{1}{n}\int_{0}^{\tau_n}\int_{(0,y)} \left(\frac{\widehat S_n(y-)}{Y(y)/n}-\frac{1}{1-G(y-)}\right)\frac{ K'(x,y)}{(1-G(x-))}dM(x)dM(y) = o_p(1).\label{eqn:Off2}
\end{align}
We only prove Equation~\eqref{eqn:Off1}, as Equation~\eqref{eqn:Off2} follows by repeating the same steps.

Define $U_n(x) = \left(\frac{\widehat S_n(x-)}{Y(x)/n}-\frac{1}{1-G(x-)}\right)$ which is predictable w.r.t.  $(\mathcal F_x)_{x\geq 0}$, and define the process $W(t)$ as
\begin{align*}
W(t) =\int_{0}^t\int_{(0,y)}U_n(x)\frac{\widehat S_n(y-)K'(x,y)}{Y(y)}dM(x)dM(y),
\end{align*}
which, by Theorem~\ref{thm:doubleMartingale}, is a square-integrable $(\mathcal F_t)$-martingale. We just need to prove that $W(\tau_n) = o_p(1)$. By the Lenglart-Rebolledo inequality, it is enough to check that the predictable variation process of $W(t)$, $\langle W \rangle(t)$, satisfies $\langle W \rangle(\tau_n) =o_p(1)$. From Theorem~\ref{thm:doubleMartingale}, we have
\begin{align}
\langle W \rangle(\tau_n) &= \int_{0}^{\tau_n}\left(\int_{(0,y)} U_n(x)K'(x,y)dM(x)\right)^2\frac{\widehat S_n(y-)^2S(y)}{Y(y)S(y-)^2}dF(y)\nonumber\\
&=O_p(1)\frac{1}{n}\int_{0}^{\tau}\left(\int_{0}^{\tau_n} U_n(x)K'(x,y)\ind_{\{x<y\}}dM(x)\right)^2
\frac{S(y)}{1-H(y-)}dF(y),\nonumber\\
&=O_p(1)\frac{1}{n}\int_{0}^{\tau}M_y^\star(\tau_n)^2\frac{S(y)}{1-H(y-)}dF(y),\label{eqn:random39183g31df}
\end{align}
where the second equality is due to Propositions \ref{prop:ProbBounds}.\ref{prop:ProbBound1} and \ref{prop:ProbBounds}.\ref{prop:ProbBound2}, and in the third equality we  define $M_y^\star(t) = \int_{0}^{t} U_n(x)K'(x,y)\ind_{\{x<y\}}dM(x)$.

We proceed to check that Equation~\eqref{eqn:random39183g31df} is $o_p(1)$.  Observe that for any fixed $y\in I_H$, $M_y^\star(t) = \int_{0}^{t} U_n(x)K'(x,y)\ind_{\{x<y\}}dM(x)$ is a square-integrable $(\mathcal{F}_t)$-martingale, thus, by Lemma~\ref{lemma:IntegralOfMartingales}, the process $Z(t) =\frac{1}{n}\int_{0}^{\tau}M_y^\star(t)^2\frac{S(y)dF(y)}{1-H(y-)}$ is an $(\mathcal{F}_t)$-submartingale. Note that $\langle W \rangle(\tau_n)=Z(\tau_n)$. We check that $Z(\tau_n)=o_p(1)$ by verifying that the compensator of $Z$, which by abusing notation we denote by $\langle Z\rangle (t)$, satisfies $\langle Z\rangle (\tau_n)=o_p(1)$. From Lemma~\ref{lemma:IntegralOfMartingales}, $\langle Z\rangle (\tau_n)$ is given by

\begin{align}
&\frac{1}{n}\int_{0}^{\tau}\int_{0}^{\tau_n} U_n(x)^2K'(x,y)^2\frac{Y(x)S(x)}{S(x-)^2}
dF(x)\frac{S(y)}{1-H(y-)}dF(y)\nonumber\\ 
&\quad=O_p(1)\int_{0}^{\tau}\int_{0}^{\tau} \ind_{\{y\leq \tau_n\}}U_n(x)^2K'(x,y)^2\frac{(1-H(x-))S(x)S(y)}{(1-H(y-))S(x-)^2}dF(x)dF(y),\label{eqn:randomxijvi42}
\end{align}
where the equality  follows from Proposition~\ref{prop:ProbBounds}.\ref{prop:ProbBound3}. We shall verify the conditions of Lemma \ref{lemma:integralConvergence2} to prove that Equation~\eqref{eqn:randomxijvi42} is $o_p(1)$. Set $R(x,y) = \frac{K'(x,y)^2S(x)S(y)}{(1-H(x-)(1-H(y-))}$ and $R_n(x,y)$ as the integrand in  Equation~\eqref{eqn:randomxijvi42}. To verify the first condition of Lemma \ref{lemma:integralConvergence2}, note that  $R_n(x)\to 0$ for each $x\in I_H$, since $U_n(x) \to 0$ by Proposition~\ref{lemma:supConverSH}. To verify  the second condition, Propositions~\ref{prop:ProbBounds}.\ref{prop:ProbBound1} and \ref{prop:ProbBounds}.\ref{prop:ProbBound2} yield  $U_n(x)= O_p(1)(1-G(x-))^{-1}$ uniformly on $x\leq \tau_n$, thus the integrand satisfies
\begin{align*}
R_n(x,y)&=\ind_{\{y\leq \tau_n\}}U_n(x)^2K'(x,y)^2\frac{(1-H(x-))S(x)S(y)}{(1-H(y-))S(x-)^2}\nonumber\\
&=O_p(1)\frac{K'(x,y)^2S(x)S(y)}{(1-H(x-))(1-H(y-))} = O_p(1)R(x,y),
\end{align*}
uniformly in $x$. Finally, the function $R(x,y) = \frac{K'(x,y)^2S(x)S(y)}{(1-H(y-))(1-H(x-))}$ is integrable due to Equation~\eqref{eqn:FiniteOpe3} of Lemma~\ref{Lemma:FiniteOpe} (recall that $K' = A_1A_2K$).
\end{proof}
\section*{Acknowledgement}
Tamara Fern\'andez was supported by the Biometrika Trust. Nicol\'as Rivera was supported by Thomas Sauerwald's ERC Starting Grant 679660 DYNAMIC MARCH.

\bibliography{ref.bib}
\appendix

\section{Proof of Lemma~\ref{Lemma:FiniteOpe}}\label{apendix:FuturePast}
In order to prove Lemma~\ref{Lemma:FiniteOpe}, we introduce the operator $R:L^2(F) \to L^2(F)$,
\begin{align*}
(Rg)(x) = \frac{1}{S(x)} \int_x^{\infty} g(t)dF(t).
\end{align*}
Note that the operator $A$ can be written as $A= \Id-R$, where $\Id$ is the identity operator. Additionally, for bivariate functions $K:\R_+^2 \to \R$, we define $R_1K$ and $R_2K$ as the operator $R$ applied on the first and second coordinate of $K$, respectively. Note that $R_1$ and $R_2$ commute

Let $X \sim F$ and $g \in L^2(F)$, then  we claim the operator $R$ satisfies that
\begin{align}
\E\left(\frac{S(X)}{S(X-)}(Rg)(X)^2\right)\leq 4\E(g(X)^2).\label{eqn:opRP1}
\end{align}
The previous equation follows from Equation (4.3) of \citet{Efron1990}, which states that
\begin{align}\label{eqn:equalityEfron}
\var(g(X)) = \E\left(\frac{S(X)}{S(X-)}(Ag)(X)^2\right).
\end{align}
Then, by using that $(Rg)^2(X) \leq 2(g(X)^2+ (Ag)(X)^2)$, we get
\begin{align}
\E\left(\frac{S(X)}{S(X-)}(Rg)(X)^2\right)
&\leq 2\E\left(\frac{S(X)}{S(X-)}g(X)^2\right)+2\E\left(\frac{S(X)}{S(X-)}(Ag)(X)^2\right)\nonumber\\
&\leq 2\E\left(g(X)^2\right)+2\var(g(X)^2)\leq 4\E(g(X)^2),\nonumber
\end{align}
where in the last step we used Equation~\eqref{eqn:equalityEfron}.

Let $\Gamma(x,y) = \frac{|K(x,y)|}{\sqrt{1-G(x-)}}$ and $\Sigma(x,y) = \frac{|K(x,y)|}{\sqrt{(1-G(x-))(1-G(y-))}}$, and notice that Conditions~\ref{assu:nondegenerate} and \ref{assu:degenerate} imply $\Gamma \in L^2(F \times F)$ and $\Sigma \in L^2(F \times F)$, respectively.

Assume  Condition~\ref{assu:nondegenerate} holds, then a simple computation shows
\begin{align}
&\int_0^{\tau} \int_0^{\tau} \frac{(R_1K)(x,y)^2S(x)}{1-H(x-)}dF(x)dF(y) \nonumber\\&\quad = \int_0^{\tau} \int_0^{\tau} \frac{S(x)}{1-H(x-)}\left(\frac{1}{S(x)}\int_x^\tau K(s,y)dF(s) \right)^2dF(x)dF(y)\nonumber\\
&\quad \leq  \int_0^{\tau} \int_0^{\tau} \frac{S(x)}{S(x-)}\left(\frac{1}{S(x)}\int_x^\tau \frac{|K(s,y)|}{\sqrt{1-G(s-)}}dF(s) \right)^2dF(x)dF(y)\nonumber\\
&\quad = \int_0^{\tau} \int_0^{\tau} \frac{S(x)}{S(x-)}(R_1\Gamma)(x,y)^2dF(x)dF(y)\nonumber\\
&\quad \leq 4\int_0^{\tau}\int_{0}^{\tau} \Gamma(x,y)^2dF(x)dF(y)<\infty,\label{eqn: Fo0}
\end{align}
where the last equation follows from Equation~\eqref{eqn:opRP1}, and Condition~\ref{assu:nondegenerate}.   Another similar computation shows

\begin{align}
&\int_0^{\tau}\int_0^{\tau} \frac{(R_2R_1K)(x,y)^2 S(x)S(y)}{(1-H(x-))S(y-)}dF(y)dF(x)\nonumber\\
&\quad\leq 4\int_0^{\tau}\int_0^{\tau} \frac{(R_1K)(x,y)^2 S(x)}{(1-H(x-))}dF(y)dF(x)\nonumber\\
&\quad\leq 16\int_0^{\tau}\int_0^{\tau} \Gamma(x,y)^2 dF(y)dF(x)<\infty \label{eqn: Fo4}
\end{align}
where the first inequality follows from Equation~\eqref{eqn:opRP1} applied on $R_2$ (i.e. applied on $y$). The second inequality is exactly Equation~\eqref{eqn: Fo0}.

 Similar computations, show that under Condition~\ref{assu:degenerate}, we have
\begin{align}
\int_0^{\tau} \int_0^{\tau} \frac{(R_1K)(x,y)^2S(x)}{(1-H(x-))(1-G(y-))}dF(x)\leq 4\int_0^{\tau}\int_{0}^{\tau} \Sigma(x,y)^2dF(x)dF(y)<\infty,\label{eqn: Fo1}
\end{align}
and that
\begin{align}
 & \int_0^{\tau} \frac{(R_2R_1K)(x,y)^2S(x)S(y)}{(1-H(x-))(1-H(y-))}dF(x)dF(y)\nonumber\\
  &\quad \leq 4\int_{0}^{\tau} \int_{0}^{\tau}\frac{S(y)}{1-H(y-)}( R_2\Gamma(x,y))^2dF(x)dF(y)\nonumber\\
  &\quad \leq 16\int_{0}^{\tau} \int_{0}^{\tau} \Sigma(x,y)^2dF(x)dF(y)<\infty.\label{eqn: Fo2}
\end{align}
From here,  under Condition~\ref{assu:nondegenerate}, Equation \eqref{eqn:FiniteOpe1} is a straightforward consequence of  Equation~\eqref{eqn: Fo0}, since
\begin{align}
&\int_0^{\tau}\int_0^{\tau}  \frac{(A_1K)^2(x,y)S(x)}{1-H(x-)}dF(x)dF(y)\nonumber\\
&\quad \leq 2\int_0^{\tau}\int_0^{\tau} \frac{(K(x,y)^2+(R_1K)(x,y)^2) S(x)}{1-H(x-)}dF(x)dF(y).\nonumber
\end{align}
Also, Equation~\eqref{eqn:FiniteOpe4} follows directly from  Equations~\eqref{eqn: Fo0} and~\eqref{eqn: Fo4} since 

\begin{align}
&\int_0^{\tau}\int_0^{\tau}  \frac{(A_1A_2K)^2(x,y)S(x)S(y)}{(1-H(x-))S(y-))}dF(x)dF(y)\nonumber\\
& \leq 4\int_0^{\tau}\int_0^{\tau}  \frac{\Phi(x,y)S(x)S(y)}{(1-H(x-))S(y-))}dF(x)dF(y),\nonumber
\end{align}
where \begin{align*}
\Phi(x,y)= K(x,y)^2+ (R_1K)(x,y)^2+ (R_2K)(x,y)^2 + (R_1R_2K)(x,y)^2.
\end{align*}

Condition~\ref{assu:degenerate}, together with Equations \eqref{eqn: Fo1} and \eqref{eqn: Fo2}, yields Equations \eqref{eqn:FiniteOpe2} and \eqref{eqn:FiniteOpe3} 
by following the exact same procedure.
\section{Proof of Theorem~\ref{thm:doubleMartingale}}\label{Appendix:DSI}

We proceed to prove Theorem~\ref{thm:doubleMartingale}. Let $h$ be $\mathcal P$-measurable. As $M(x)$ is the difference between two right-continuous increasing processes, we have
\begin{align}
\int_{C_t} h(x,y)dM(x)dM(y) = \int_0^t \int_{(0,y)}h(x,y)dM(x)dM(y).\nonumber
\end{align}
We proceed to prove that the process
$\phi(y) = \int_{(0,y)}h(x,y)dM(x)$ is predictable with respect to the sigma-algebra $(\mathcal F_y)_{y\geq 0}$. For this, it is enough to verify the claim for elementary functions of $\mathcal P$, and then we extend the result to general functions in $\mathcal P$.

If $h(x,y)=X\ind_{\{(x,y)\in(a_1,b_1]
\times(a_2,b_2]\}}$ with $X\in\mathcal{F}_{a_2}$ and $0\leq a_1\leq b_1\leq a_2\leq b_2$, then 
\begin{align}
\int_{(0,y)}h(x,y)dM(x) =X\ind_{(a_2,b_2]}(y)\int_{(0,y)}  \ind_{(a_1,b_1]}(x)dM(x),\nonumber
\end{align}
which is predictable with respect to $(\mathcal F_y)_{y\geq 0}$ since both processes, $X\ind_{(a_2,b_2]}(y)$ and $\int_{(0,y)}  \ind_{(a_1,b_1]}(x)dM(x)$, are adapted to $(\mathcal F_y)_{y\geq 0}$ and  left-continuous. For the first process note that it is important that $X \in \mathcal F_{a_2}$ to ensure it is adapted, and for the second one it is key that we are integrating on $(0,y)$ instead of $(0,y]$ to ensure it is left-continuous. Therefore, the process
\begin{align}
Z(t)=\int_0^t\int_{(0,y)}h(x,y)dM(x)dM(y)\nonumber
\end{align}
is the integral of a predictable process, and thus $Z_t$ is an $(\mathcal F_t)$-martingale.  By using Equation~\eqref{eqn:DoubleMcond2} together with Lebesgue Dominated Convergence theorem,  we extend the result to  general functions $h$ of the predictable sigma algebra $\mathcal P$. From Equation~\eqref{eqn:doubleMComcond2}, we get that $Z(t)$ is a square-integrable process, and its predictable variation process is given by 
\begin{align}
\langle Z \rangle(t) = \int_0^t \left(\int_{(0,y)}h(x,y)dM(x) \right)^2 \frac{Y(y)S(y)}{S(y-)^2}dF(y).\nonumber
\end{align}

\section{Properties of {$J$}{}}\label{sec:UStatsRepBeta}
In this section we show
\begin{enumerate}
\item[i.] $\E(J((X_1,\Delta_1),(x,r)))=0$ for any $(x,r)\in I_H \times \{0,1\}$.
\item[ii.] $\E(J((X_1,\Delta_1),(X_1,\Delta_1)))=\int_{0}^{\tau}\frac{S(x)}{1-H(x-)}K'(x,x)dF(x)$
\item[iii.]$\E(J((X_1,\Delta_1),(X_2,\Delta_2))^2)=\int_{0}^{\tau} \int_0^{\tau} \frac{S(x)S(y)}{(1-H(x-))(1-H(y-))}K'(x,y)^2dF(x)dF(y)$
\end{enumerate}

We start with item i. To ease notation, define 
$\tilde K(x,y) = K'(x,y)/((1-G(x-)(1-G(y-))$ and let $dm(s) = r \delta_{x}(ds)-\ind_{\{x\geq s\}}d\Lambda(s)$, then
\begin{align}
J((X_1,\Delta_1),(x,r)) = \int_{0}^{X_1} \left( \int_{0}^{x} \tilde K(s,t)dm(s)\right)dM_1(t).\nonumber
\end{align}
As the term inside the parenthesis is a deterministic function of $t$, the previous integral is just a stochastic integral with respect to the zero mean martingale $M_1$, then by the Optional Stopping Theorem, its expected value is 0.

We continue with item ii. Observe that
\begin{align}
&\E(J((X_1,\Delta_1),(X_1,\Delta_1))\nonumber\\
&\quad=\E\left(\int_{0}^{\tau_n}\int_{0}^{\tau_n} \frac{K'(x,y)}{(1-G(x-))(1-G(y-))}dM_1(x)dM_1(y)\right)\nonumber\\
&\quad= \E\left(\int_{0}^{\tau_n} \frac{K'(x,x)}{(1-G(x-))^2}(dM_1(x))^2\right),\nonumber
\end{align}
where the last equality follows from the fact that the integral in the off-diagonal, that is, for $x\neq y$, defines a zero-mean martingale (see Definition~\ref{def:predicDsigmaAlgebra} and Theorem~\ref{thm:doubleMartingale}). Then, 
\begin{align*}
 \E\left(\int_{0}^{\tau_n} \frac{K'(x,x)}{(1-G(x-))^2}(dM_1(x))^2\right)&=\E\left(\int_{0}^{\tau}\frac{K'(x,x)}{(1-G(x-))^2}d[M_1](x) \right)\nonumber\\
 &= \E\left(\int_{0}^{ \tau}\frac{K'(x,x)}{(1-G(x-))^2}d\langle M_1 \rangle(x) \right)\nonumber\\
 &= \E\left(\int_{0}^{ \tau}\frac{K'(x,x)}{(1-G(x-))^2}\frac{S(x)Y_1(x)}{S(x-)^2}dF(x) \right)\nonumber\\
 &=\int_{0}^{\tau}\frac{S(x)}{1-H(x-)}K'(x,x)dF(x),\nonumber
\end{align*}
where the first and second equalities are due to the properties of $[M_1]$ (see beginning of Section~\ref{eqn:sectionIntegralonDiagonal}). The third equality follows from the definition of $\langle M_1 \rangle$, and the last equality follows by interchanging the integral and expectation.

We finalise with item iii. Let $\tilde K(x,y) = K'(x,y)/((1-G(x-)(1-G(y-))$, by conditional expectation, it holds
\begin{align}
&\E(J((X_1,\Delta_1),(X_2,\Delta_2))^2) \nonumber\\
&\quad= \E(\E(J((X_1,\Delta_1),(X_2,\Delta_2))^2|(X_2,\Delta_2)))\nonumber\\
&\quad=\E\left(\E\left(\left(\int_{0}^{X_1} \int_{0}^{X_2} \tilde K(x,y) dM_2(y)dM_1(x)\right)^2 \given (X_2, \Delta_2)\right)\right),\nonumber\\
&\quad=\E\left(\E\left(\int_{0}^{X_1} \left(\int_{0}^{X_2} \tilde K(x,y) dM_2(y)\right)^2d\langle M_1 \rangle(x) \given (X_2,\Delta_2)\right)\right),\label{eqn:randomeq983godm9301}
\end{align}
\sloppy{where the last equality follows  from the fact that, conditioned on $(X_2,\Delta_2)$,  $\int_{0}^{t} \int_{0}^{X_2} \tilde K(x,y) dM_2(y)dM_1(x)$ is an $(\mathcal F_t)$-martingale.}

Replacing with the value of $\langle M_1(x)\rangle$ in Equation \eqref{eqn:randomeq983godm9301}, we have
\begin{align}
&\E(J((X_1,\Delta_1),(X_2,\Delta_2))^2)\nonumber\\
&\quad=\E\left(\int_{0}^{X_1} \left(\int_{0}^{X_2} \tilde K(x,y) dM_2(y)\right)^2\frac{S(x)Y_1(x)}{S(x-)^2}dF(x) \right)\nonumber\\ 
&\quad=\E\left(\int_{0}^{\tau} \left(\int_{0}^{X_2} \tilde K(x,y) dM_2(y)\right)^2\frac{S(x)Y_1(x)}{S(x-)^2}dF(x) \right)\nonumber\\ 
&\quad=\int_{0}^{\tau} \E\left(\left(\int_{0}^{X_2} \tilde K(x,y) dM_2(y)\right)^2\right)\frac{S(x)(1-H(x-))}{S(x-)^2}dF(x) \nonumber\\
&\quad=\int_{0}^{\tau} \E\left(\int_{0}^{\tau} \tilde K(x,y)^2 \frac{S(y)Y_2(y)}{S(y-)^2}dF(y)\right)(1-H(x-))\frac{S(x)}{S(x-)^2}dF(x)\nonumber\\
&\quad=\int_{0}^{\tau} \int_0^{\tau} \frac{K'(x,y)^2S(x)S(y)}{(1-H(x-))(1-H(y-))}dF(x)dF(y),\nonumber
\end{align}
where the third equality follows from the fact that $(X_1,\Delta_1)$ and $(X_2,\Delta_2)$ are independent and by Fubini's Theorem. The fourth inequality follows by noticing that, for fixed $x$, $\int_0^t \bar K(x,y)dM_2(y)$ is a square-integrable $(\mathcal F_t)$-martingale with predictable variation process given by $\int_0^t \bar K(x,y)^2 \frac{S(y-)}{S(y-)^2}Y_2(y)dF(y)$. 
\section{Proof of Lemma~\ref{lemma:diagKnonDege}}\label{Appendix:SecDiag}
Equation (3.41) of \citet{aalen2008survival} states that
\begin{align}
\Delta \widehat F_n(x) = \widehat S_n(x-)\frac{\Delta N(x)}{Y(x)},\nonumber
\end{align}
hence a Kaplan-Meier weight $W_i$ for an uncensored observation $X_i$ equals $\Delta \widehat F_n(X_i)$ divided by all the uncensored observations that fall exactly in $X_i$, i.e., the weight $W_i$ associated with $X_i$ equals $\frac{\Delta \widehat F_n(X_i)}{\Delta N(X_i)}=\frac{\widehat S_n(X_i-)}{Y(X_i)}$. Then
\begin{align}
\sum_{i=1}^n K(X_i,X_i)W_i^2&=\sum_{i=1}^n K(X_i,X_i)\left(\frac{\widehat S_n(X_i-)}{Y(X_i)}\right)^2\Delta_i \nonumber\\
&= \int_0^{\tau_n} K(x,x)\left(\frac{\widehat S_n(x-)}{Y(x)}\right)^2dN(x).\label{eqn:randomfle9391}
\end{align}
We will first prove that $\sqrt{n}\sum_{i=1}^n K(X_i,X_i)W_i^2 = o_p(1)$ under Condition~\ref{assu:nondegenerate}. Note that the process 
$\int_0^{t}|K(x,x)|\left(\frac{\widehat S_n(x-)}{Y(x)}\right)^2dN(x)$ is an $(\mathcal F_t)$-submartingale, with compensator given by $Z(t) = \int_0^{t}|K(x,x)|\left(\frac{\widehat S_n(x-)}{Y(x)}\right)^2\frac{Y(x)}{S(x-)}dF(x)$. By the Lenglart-Rebolledo inequality, it is enough to prove that $\sqrt{n}Z(\tau_n) = o_p(1)$.
An application of Propositions~\ref{prop:ProbBounds}.\ref{prop:ProbBound1} and~\ref{prop:ProbBounds}.\ref{prop:ProbBound2}  shows that
\begin{align*}
\sqrt{n}Z(\tau_n) &= \int_0^{\tau_n}|K(x,x)|\left(\frac{\widehat S_n(x-)}{Y(x)}\right)^2\frac{Y(x)}{S(x-)}dF(x)\nonumber\\
&=O_p(1)\int_0^{\tau}\frac{\ind_{\{x\leq \tau_n\}}}{\sqrt{Y(x)}}|K(x,x)|\sqrt{\frac{ S(x-)}{(1-G(x-)}}dF(x).\nonumber
\end{align*}
Finally, $\int_0^{\tau}\frac{\ind_{\{x\leq \tau_n\}}}{\sqrt{Y(x)}}|K(x,x)|\sqrt{\frac{ S(x-)}{1-G(x-)}}dF(x) = o_p(1)$ by the Dominated Convergence Theorem since $Y(x) \to \infty$ for each $x \in I_H$, $Y(x)\geq 1$ for $x \leq \tau_n$, and $\int_0^{\tau} |K(x,x)|\sqrt{\frac{ S(x-)}{(1-G(x-)}}dF(x)<\infty$ by hypothesis.

We now prove that $n\sum_{i=1}^n K(X_i,X_i)W_i^2 = \int_0^{\tau} \frac{K(x,x)}{1-G(x-)}dF(x)+o_p(1)$. We start by proving the intermediate step
\begin{align}
\frac{1}{n}\int_0^{\tau_n} |K(x,x)|\left|\left(\frac{\widehat S_n(x-)}{(Y(x)/n)}\right)^2-\frac{1}{(1-G(x-))^2}\right| dN(x)&=o_p(1).
\label{eqn:random192384isdvijsd}
\end{align}
Denote by $U_n(x) = \left|\left(\frac{\widehat S_n(x-)}{(Y(x)/n)}\right)^2-\frac{1}{(1-G(x-))^2}\right|$ which is predictable w.r.t. $(\mathcal F_x)_{x\geq 0}$. Then, observe that $\frac{1}{n}\int_0^{t} |K(x,x)|U_n(x) dN(x)$ is an $(\mathcal F_t)$-submartingale with compensator, evaluated at $\tau_n$, given by
\begin{align}
&\frac{1}{n}\int_0^{\tau_n} |K(x,x)| U_n(x) \frac{Y(x)}{S(x-)}dF(x)\nonumber\\
&\quad=O_p(1)\int_0^{\tau} \ind_{\{x\leq \tau_n\}} |K(x,x)|U_n(x) (1-G(x-))dF(x),\label{eqn:randombir9184fcb}
\end{align}
where the equality holds by Proposition~\ref{prop:ProbBounds}.\ref{prop:ProbBound3}.
We claim that Equation \eqref{eqn:randombir9184fcb} equals $o_p(1)$ by Lemma~\ref{lemma:integralConvergence2}, whose conditions we proceed to verify. Set $R(x) = \frac{|K(x,x)|}{1-G(x-)}$, which is integrable, and set $R_n(x)$ as the integrand in Equation \eqref{eqn:randombir9184fcb}. For the first condition of Lemma~\ref{lemma:integralConvergence2}, note that $U_n(x) \to 0$ for all $x$, due to Proposition~\ref{lemma:supConverSH}, hence $R_n(x) \to 0$. For the second condition, Propositions~\ref{prop:ProbBounds}.\ref{prop:ProbBound1} and~\ref{prop:ProbBounds}.\ref{prop:ProbBound2} yield $U_n(x) = O_p(1)(1-G(x-))^{-2}$ uniformly in $x\leq\tau_n$, and thus $R_n(x) =O_p(1) R(x)$. 

Finally, the Lenglart-Rebolledo Inequality, we get that Equation~\eqref{eqn:random192384isdvijsd} holds true. Combining \eqref{eqn:randomfle9391} and \eqref{eqn:random192384isdvijsd} yields
\begin{align*}
n \sum_{i=1}^n K(X_i,X_i)W_i^2 = \frac{1}{n}\sum_{i=1}^n \frac{K(X_i,X_i)}{(1-G(X_i-))^2}\Delta_i+o_p(1),
\end{align*}
hence, by the Law of Large numbers we obtain that $n\sum_{i=1}^n K(X_i,X_i)W_i^2 = \int_0^{\tau} \frac{K(x,x)}{1-G(x-)}dF(x) +o_p(1)$.

\end{document}

%% file: defsProb.tex
\theoremstyle{plain}
\newtheorem{theorem}{Theorem}[section]
\newtheorem{lemma}[theorem]{Lemma}
\newtheorem{corollary}[theorem]{Corollary}
\newtheorem{proposition}[theorem]{Proposition}

\theoremstyle{definition}
\newtheorem{definition}[theorem]{Definition}
\newtheorem{example}[theorem]{Example}%
\theoremstyle{remark}

\theoremstyle{plain}
\newtheorem{condition}[theorem]{Condition}

\newcommand{\proofstring}{Proof}
\renewenvironment{proof}[1][]{\noindent\textbf{\ifthenelse{\equal{#1}{}}{\proofstring:}{~#1:}}\xspace}{\hfill$\blacksquare$\medskip\par}

\newcommand{\R}{\mathds{R}}

\newcommand{\E}{\mathds{E}}
\newcommand{\var}{\mathds{V}ar}
\newcommand{\Var}{\var}

\newcommand{\Prob}{\mathds{P}}
\newcommand{\Ind}{\mathds{1}}
\newcommand{\ind}{\Ind}

\newcommand{\given}{\middle|}